\documentclass[letterpaper,     
11pt,                   
dvipsnames,
]{article}%
\usepackage[margin=1in]{geometry}
\usepackage{cancel}
\usepackage[normalem]{ulem}

\PassOptionsToPackage{utf8}{inputenc}
\usepackage{inputenc}
\usepackage{csquotes}

\usepackage{amsmath}
\usepackage[USenglish]{babel}
\usepackage{amsthm, amsfonts, amssymb}
\usepackage{mathtools}
\usepackage{aligned-overset}
\usepackage{doi}
\usepackage{wrapfig}
\usepackage[noend, algosection]{algorithm2e}
\usepackage{paralist}
\usepackage{soul}

\usepackage[shortlabels]{enumitem}
\usepackage{thmtools}
\usepackage{comment}
\usepackage[dvipsnames]{xcolor}

\usepackage{tikz}

\hypersetup{
pdftitle={The Strong Nine Dragon Tree Conjecture is True for d <= 2(k+1) },    
pdfsubject={Preprint},
pdfauthor={Sebastian Mies, Benjamin Moore},    
plainpages=false,           
colorlinks=false,           
linkcolor=ctcolormain,      
citecolor=ctcolormain,      
pdfborder={0 0 0},          
breaklinks=true,            
bookmarksnumbered=true,     %
bookmarksopen=true          %
}

\graphicspath{{images/}}

\definecolor{lightBlue}{RGB}{136, 247, 244}

\def\pathIn#1#2#3{P^{#1}(#2, #3)}
\def\blueTree{\mathcal B_b}
\def\redForest{\mathcal R}

\def\barT{\bar{\mathcal T}}
\def\pathInBlueTree#1#2#3{\pathIn{\blueTree(#1)}{#2}{#3}}

\usepackage{todonotes}
\usepackage[style=numeric-comp, maxbibnames=99]{biblatex}

\author{Sebastian Mies\thanks{Institute of Computer Science, Johannes Gutenberg University Mainz, email: smies@students.uni-mainz.de}   \, and Benjamin Moore\thanks{Institute of Science and Technology Austria, email: Benjamin.Moore@ist.ac.at. Benjamin Moore is supported by ERC Starting Grant “RANDSTRUCT” No. 101076777 and appreciates the gracious support.}}

\title{The Strong Nine Dragon Tree Conjecture is True for $d \leq 2(k + 1)$ }

\bibliography{bib.bib}

\DeclarePairedDelimiter\ceil{\lceil}{\rceil}
\DeclarePairedDelimiter\floor{\lfloor}{\rfloor}

\newcommand\fracArb{\gamma}
\newcommand\explSG{H_{\mathcal T^*}}
\newcommand\density{\frac{d}{d+k+1}}

\DeclareRobustCommand\caseOne{\overset{(1, b)}{\longrightarrow}}
\DeclareRobustCommand\caseTwo{\overset{(2, b)}{\longrightarrow}}

\newcommand\inOneToK{\in \{1, \ldots, k\}}

\declaretheoremstyle[
        spaceabove=-2em,
        spacebelow=6pt,
        headfont=\normalfont\itshape,
        postheadspace=1em,
        qed=\qedsymbol
    ]{proofStyle}
\declaretheoremstyle[
        spaceabove=1em,
        spacebelow=6pt,
        headfont=\normalfont\itshape,
        postheadspace=1em,
        qed=\qedsymbol
    ]{proofOfFinalLemmaStyle}
\declaretheoremstyle[
        spaceabove=-2em,
        spacebelow=6pt,
        headfont=\normalfont\itshape,
        postheadspace=1em,
        qed=\hfill \textit{\color{gray}(End of proof of the claim) }$\;\blacksquare$
    ]{proofInProofStyle}

\declaretheorem[name={Proof},style=proofInProofStyle,unnumbered,
]{proofInProof}

\declaretheorem[name={\textbf{Proof of Lemma \ref{lemma:densityOfKC}}},style=proofOfFinalLemmaStyle,unnumbered,
]{proofOfFinalLemma}
\newtheorem{thm}{Theorem}[section] 
\newtheorem{lemma}[thm]{Lemma}
\newtheorem{conj}[thm]{Conjecture}
\newtheorem{definition}[thm]{Definition}
\newtheorem{obs}[thm]{Observation}
\newtheorem{corollary}[thm]{Corollary}
\newtheorem*{claimUnnumbered}{Claim}
\newtheorem{claim}[thm]{Claim}
\newtheorem{notation}[thm]{Notation}
\newtheorem*{ack}{Acknowledgements}
\date{}

\def\lowerMadBoundWithAlpha{\frac{\ell(k+1) + \alpha}{(\ell + 1)(k+1) + \alpha}}
\def\densNDT{\frac{d}{d+k+1}}
\def\equivModKPlusOne#1#2{#1 \equiv #2 \mod (k+1)}

\begin{document}
\definecolor{lightBlue}{RGB}{136, 247, 244}

\def\scale{0.46}
\def\lengthEdge{1.5}
\def\lengthP{8*\lengthEdge}
\def\offsetNames{-1.6em}
\def\xOfTPred{0.25*\lengthP}
\def\yOfZPred{-\lengthP + \xOfTPred}
\def\bendS{10}
\def\bendNyS{70}
\def\bendxNx{45}
\def\bendUXLong{20}
\def\bendUXShort{40}
\def\lengthKxCaseTwo{(\lengthP - \lengthEdge) / 2}
\def\lengthKxCaseOne{(\lengthP - \lengthMiddle - 2*\lengthEdge)/2}
\def\uxPrime{(u) edge[bend right=\bendUXLong] (x')}
\def\xPrimeU{(x') edge[bend left=\bendUXLong] (u)}

\def\vyPrime{(v) edge[bend left=\bendUXLong] (y')}
\def\yPrimeV{(y') edge[bend right=\bendUXLong] (v)}

\def\vyShort{(v) edge[bend left=\bendUXShort] (y)}
\def\yvShort{(y) edge[bend right=\bendUXShort] (v)}

\def\vyLong{(v) edge[bend left=\bendUXLong] (y)}
\def\yvLong{(y) edge[bend right=\bendUXLong] (v)}

\def\vyBelowPShort{(v) edge[bend right=\bendUXShort] (y)}
\def\yvBelowPShort{(y) edge[bend left=\bendUXShort] (v)}

\def\vyBelowPLong{(v) edge[bend right=\bendUXLong] (y)}
\def\yvBelowPLong{(y) edge[bend left=\bendUXLong] (v)}

\def\bendInNyPrimeS{115}
\def\bendOutNyPrimeS{75}
\def\nyPrimeS{(ny') edge[bend left, in=\bendInNyPrimeS, out=\bendOutNyPrimeS] (s)}
\def\sNyPrime{(s) edge[bend right, in=180+\bendOutNyPrimeS, out=180+\bendInNyPrimeS] (ny')}

\def\sNxPrime{(s) edge[bend left=\bendS] (nx')}
\def\nxPrimeS{(nx') edge[bend right=\bendS] (s)}

\def\xNxPrime{(x) edge[bend left=\bendxNx] (nx')}
\def\nxPrimeX{(nx') edge[bend right=\bendxNx] (x)}

\def\yNyPrime{(y) edge[bend right=\bendxNx] (ny')}
\def\nyPrimeY{(ny') edge[bend left=\bendxNx] (y)}

\def\bendXPrimePrimeYOut{30}
\def\bendXPrimePrimeYIn{150}
\def\xPrimePrimeY{(x'') edge[bend left, in=\bendXPrimePrimeYIn, out=\bendXPrimePrimeYOut] (y)}
\def\yXPrimePrime{(y) edge[bend left, in=180+\bendXPrimePrimeYOut, out=180+\bendXPrimePrimeYIn] (x'')}

\def\xs{(x) edge[bend left=\bendS] (s)}

\def\uPrimeYPrime{(u') edge[bend left=\bendUXLong] (y')}

\def\bendUPrimeWYPrime{30}
\def\bendWYPrimeOne{10}
\def\bendWYPrimeTwo{10}
\def\drawWYPrime{
	\draw[blueDottedUndir]
		(w) edge[bend left=\bendWYPrimeOne] (wy')
	;
	\draw[blueDotted]
		(wy') edge[bend left=\bendWYPrimeTwo] (y')
	;
}
\def\drawYPrimeW{
	\draw[blueDottedUndir]
		(y') edge[bend right=\bendWYPrimeOne] (wy')
	;
	\draw[blueDotted]
		(wy') edge[bend right=\bendWYPrimeTwo] (w)
	;
}
\def\transpBelowXToS{(transpBelowX) edge[bend left=30] (s)}
\def\sToTranspBelowX{(s) edge[bend right=30] (transpBelowX)}

\def\nzPrimeTranspBelowX{(nz') edge[bend left=50] (transpBelowX)}
\def\transpBelowXToNzPrime{(transpBelowX) edge[bend right=50] (nz')}

\def\nyXPrime{(ny) edge[bend right=10] (x')}
\def\xPrimeNy{(x') edge[bend left=10] (ny)}

\def\vz{(v) edge[bend right=30] (z)}
\def\zv{(z) edge[bend left=30] (v)}

\def\us{(u) edge[bend left=40] (s)}
\def\su{(s) edge[bend right=40] (u)}

\def\vPrimeYPrime{(v') edge[bend left=\bendUXLong] (y')}

\def\angleOne{310}
\def\angleTwo{345}
\def\angleThree{30}

\def\myTikZ#1{
	\begin{tikzpicture}[scale=\scale, baseline=-1mm, -,]
		#1
	\end{tikzpicture}
}

\def\doubleFigure#1#2#3#4{
    \begin{figure}[!htb]
        \minipage{0.49\textwidth}
        \raggedright
        #1
        \endminipage\hfill
        \minipage{0.49\textwidth}
        \raggedleft
        #2
        \endminipage
        \caption{#3}
        #4
    \end{figure}
}

\def\singleFigure#1#2#3{
    \begin{figure}[!htb]
        \center
        #1
        \caption{#2}
        #3
    \end{figure}
}
	
\tikzset{
    string/.style={},
    redDot/.style={circle, draw, red, fill=red, inner sep=1.5pt},
    blackDot/.style={circle, draw, fill, inner sep=1.5pt},
    transpBetweenUndir/.style={inner sep=0.1pt},
    transp/.style={inner sep=1.5pt},
    blueEdge/.style={line width = 0.5mm, cyan, -stealth},
    redEdge/.style={red, line width=0.5mm},
    redDotted/.style={red, densely dotted, line width=0.5mm},
    blueDotted/.style={cyan, loosely dotted, line width=0.5mm, -stealth},
    blueDottedUndir/.style={cyan, loosely dotted, line width=0.5mm},
}

\def\nodeY{
	\node[redDot, label={[label distance=\offsetNames]90:$y$}] (y) at (\lengthP, 0) {};
}

\def\nodeS{
	\node[transp] (s) at (-0.8*\lengthEdge, 0.2*\lengthEdge) {};
}

\def\nodePrime#1#2#3{ 
	\node[redDot, label={[label distance=1.4*\offsetNames]90 + #2*225:$#3'$}] (#3') at ({#2*#1}, \lengthEdge) {};
}
\def\nodeAndHalfEdgeTranspAroundInteresting#1#2#3{ 
	\node[transpBetweenUndir] (#3AroundInteresting) at ({#1 + #2*2.1*\lengthEdge}, 1.5*\lengthEdge) {};
	\draw[blueDottedUndir]
		(#3'') edge[bend left=#2*90] (#3AroundInteresting)
	;
}

\def\edgesInterestingXInTxWithoutS{
	\draw[redEdge]
		(x) edge (x')
	;
	\draw[blueEdge]
		(x') edge (x'')
		(nx') edge (x')	
	;
	\draw[blueDotted]
		\uxPrime
		\xNxPrime
	;
}

\def\edgesInterestingXInTxWithS{
	\draw[redEdge]
		(x) edge (x')
	;
	\draw[blueEdge]
		(x') edge (x'')
		(nx') edge (x')	
	;
	\draw[blueDotted]
		\xs
		\sNxPrime
	;
}

\def\edgesTxWithS{
	\edgesInterestingXInTxWithS
	\draw[redEdge]
		(y') edge (ny')
	;
	\draw[blueEdge]
		(y) edge (y')
		(y') edge (y'')
	;
	\draw[blueDotted]
		\nyPrimeS
		\uxPrime
	;
}

\def\interestingNeighbourBasic#1#2#3{ 
	\node[redDot, label={[label distance=\offsetNames]90:$#3$}] (#3) at ({#2*#1}, 0) {};
	\nodePrime{#1}{#2}{#3}

	\node[redDot, label={[label distance=\offsetNames]270:$n_{#3'}$}] (n#3') at ({#2*#1 + #2*\lengthEdge}, \lengthEdge) {};
	\node[redDot, label={[label distance=1.2*\offsetNames]90 + #2*90:$#3''$}] (#3'') at ({#2*#1}, 2*\lengthEdge) {};
	
	\node[transp] (c#31) at ({#2*#1 + #2*\lengthEdge + #2*cos(\angleOne)*\lengthEdge}, {\lengthEdge + sin(\angleOne)*\lengthEdge}) {};
	\node[transp] (c#32) at ({#2*#1 + #2*\lengthEdge + #2*cos(\angleTwo)*\lengthEdge}, {\lengthEdge + sin(\angleTwo)*\lengthEdge}) {};
	\node[transp] (c#33) at ({#2*#1 + #2*\lengthEdge + #2*cos(\angleThree)*\lengthEdge}, {\lengthEdge + sin(\angleThree)*\lengthEdge}) {};
	
	\draw[redDotted]
	(n#3') edge (c#31)
	(n#3') edge (c#32)
	(n#3') edge (c#33)
	;
}

\def\xWithInterestingNeighbourBasic{
	\interestingNeighbourBasic{0}{-1}{x}
}

\def\yWithInterestingNeighbourBasic{
	\interestingNeighbourBasic{\lengthP}{1}{y}
}

\def\caseTwoBasic#1#2{

	\xWithInterestingNeighbourBasic

	\node[redDot, label={[label distance=\offsetNames]90:$#1$}] (u) at ({\lengthKxCaseTwo}, 0) {};
	\node[redDot, label={[label distance=\offsetNames]90:$#2$}] (v) at ({\lengthKxCaseTwo + \lengthEdge}, 0) {};
	
	\draw[redDotted]
		(x) edge (u)
		(v) edge (y)
	;
}

\def\caseTwoBasicSwapLemma{
	\yWithInterestingNeighbourBasic
	\nodeS
	
	\caseTwoBasic{u}{v}
	\draw[blueEdge]
		(y') edge (y'')
	;
}

\def\caseTwoCaseDef{\myTikZ{
	\nodeY
	\nodePrime{\lengthP}{1}{y}
	\caseTwoBasic{x_i}{x_{i+1}}
	\edgesInterestingXInTxWithoutS
	\draw[redEdge]
		(u) edge (v)
	;
	\draw[blueEdge]
		(y) edge (y')
	;
	\draw[blueDotted]
		\vyLong
	;
}}

\def\caseTwoStart{\myTikZ{
	\caseTwoBasicSwapLemma
	\edgesTxWithS
	\draw[redEdge]
		(u) edge (v)
	;
	\draw[blueEdge]
	;
	\draw[blueDotted]
		\xPrimePrimeY
		\vyBelowPLong
	;
}}

\def\caseTwoEnd{\myTikZ{
	\caseTwoBasicSwapLemma
	\draw[redEdge]
		(x) edge (x')
		(y) edge (y')
		(x') edge (x'')
	;
	\draw[blueEdge]
		(v) edge (u)
		(x') edge (nx')
		(ny') edge (y')
	;
	\draw[blueDotted]
		\uxPrime
		\yvBelowPLong
		\xPrimePrimeY
		\xs
		\nxPrimeS
		\sNyPrime
	;
}}

\def\caseOneBasic#1#2#3#4{
	\def\lengthMiddle{1.5*\lengthEdge}

	\xWithInterestingNeighbourBasic
	
	\node[redDot, label={[label distance=\offsetNames]90:$#1$}] (u) at ({\lengthKxCaseOne}, 0) {};
	\node[redDot, label={[label distance=\offsetNames]90:$#2$}] (u') at ({\lengthKxCaseOne + \lengthEdge}, 0) {};
	
	\node[redDot, label={[label distance=\offsetNames]90:$#3$}] (v') at ({\lengthKxCaseOne + \lengthEdge + \lengthMiddle}, 0) {};
	\node[redDot, label={[label distance=\offsetNames]90:$#4$}] (v) at ({\lengthKxCaseOne + 2*\lengthEdge + \lengthMiddle}, 0) {};
	
	\draw[redDotted]
		(x) edge (u)
		(v) edge (y)
		(u') edge (v')
	;
}

\def\caseOneBasicSwapLemmas{
	\yWithInterestingNeighbourBasic
	\caseOneBasic{u}{u'}{v'}{v}
}

\def\caseOneCaseDef{\myTikZ{
	\nodeY
	\nodePrime{\lengthP}{1}{y}
	\caseOneBasic{x_i}{x_{i+1}}{x_{j-1}}{x_j}
	\edgesInterestingXInTxWithoutS
	\draw[redEdge]
		(u) edge (u')
		(v) edge (v')
	;
	\draw[blueEdge]
		(y) edge (y')
	;
	\draw[blueDotted]
		\vyShort
	;
}}

\def\caseOneClarificationForVVPrime{\myTikZ{
	\caseOneBasicSwapLemmas
	\edgesInterestingXInTxWithoutS
	\draw[redEdge]
		(ny') edge (y')
		(u) edge (u')
		(v) edge (v')
	;
	\draw[blueEdge]
		(y) edge (y')
		(y') edge (y'')
	;
	\draw[blueDotted]
		\vyShort
		\nyPrimeY
	;
}}

\def\caseOneSimpleOneStart{\myTikZ{
	\caseOneBasicSwapLemmas
	\edgesInterestingXInTxWithoutS
	\draw[redEdge]
		(ny') edge (y')
		(u) edge (u')
		(v) edge (v')
	;
	\draw[blueEdge]
		(y) edge (y')
		(y') edge (y'')
	;
	\draw[blueDotted]
		\uPrimeYPrime
		\vyShort
		\nyPrimeY
	;
}}

\def\caseOneSimpleOneEnd{\myTikZ{
	\caseOneBasicSwapLemmas
	\draw[redEdge]
		(x) edge (x')
		(x') edge (x'')
		(v) edge (v')
		(y) edge (y')
	;
	\draw[blueEdge]
		(ny') edge (y')
		(u) edge (u')
		(y') edge (y'')
		(nx') edge (x')
	;
	\draw[blueDotted]
		\xNxPrime
		\uPrimeYPrime
		\vyShort
		\yNyPrime
		\xPrimeU
	;
}}

\def\caseOneSimpleTwoStart{\myTikZ{
	\yWithInterestingNeighbourBasic
	\caseOneBasic{u}{u'}{v_{i-1}}{v_i}
	\edgesInterestingXInTxWithoutS
	\draw[redEdge]
		(ny') edge (y')
		(u) edge (u')
		(v) edge (v')
	;
	\draw[blueEdge]
		(y) edge (y')
		(y') edge (y'')
	;
	\draw[blueDotted]
		\vyPrime
		\nyPrimeY
	;
}}

\def\caseOneSimpleTwoEnd{\myTikZ{
	\yWithInterestingNeighbourBasic
	\caseOneBasic{u}{u'}{v_{i-1}}{v_i}
	\draw[redEdge]
		(u) edge (u')
		(y) edge (y')
		(y') edge (y'')
		(x) edge (x')
	;
	\draw[blueEdge]
		(ny') edge (y')
		(nx') edge (x')
		(x') edge (x'')
		(v) edge (v')
	;
	\draw[blueDotted]
		\uxPrime
		\yPrimeV
		\yNyPrime
		\xNxPrime
	;
}}

\def\caseOneBasicUgliestCase{
	\def\lengthMiddle{\lengthEdge}
	\def\lengthKxUgliestCase{(\lengthP - 3*\lengthEdge - 2*\lengthMiddle) / 2}
	
	\xWithInterestingNeighbourBasic
	\nodeS
	\yWithInterestingNeighbourBasic
	
	\node[redDot, label={[label distance=\offsetNames]90:$u$}] (u) at ({\lengthKxUgliestCase}, 0) {};
	\node[redDot, label={[label distance=\offsetNames]90:$u'$}] (u') at ({\lengthKxUgliestCase + \lengthEdge}, 0) {};
	
	\node[redDot, label={[label distance=\offsetNames]90:$w$}] (w) at ({\lengthKxUgliestCase + \lengthEdge + \lengthMiddle}, 0) {};
	\node[redDot, label={[label distance=\offsetNames]90:$w'$}] (w') at ({\lengthKxUgliestCase + 2*\lengthEdge + \lengthMiddle}, 0) {};
	
	\node[redDot, label={[label distance=\offsetNames]90:$v'$}] (v') at ({\lengthP - \lengthKxUgliestCase - \lengthEdge}, 0) {};
	\node[redDot, label={[label distance=\offsetNames]90:$v$}] (v) at ({\lengthP - \lengthKxUgliestCase}, 0) {};
	
	\node[transp] (wy') at ({\lengthP - \lengthKxUgliestCase - \lengthEdge}, 0.8*\lengthEdge) {};
	
	\draw[redDotted]
		(x) edge (u)
		(u') edge (w)
		(w') edge (v')
		(v) edge (y)
	;
	\draw[redEdge]
		(u) edge (u')
		(v) edge (v')
	;
	\draw[blueDotted]
		(u') edge[bend left=\bendUPrimeWYPrime] (wy')
		\xPrimePrimeY
	;
}

\def\caseOneUgliestCaseStart{\myTikZ{
	\caseOneBasicUgliestCase
	\edgesTxWithS
	\drawWYPrime
	\draw[redEdge]
		(w) edge (w')
	;
	\draw[blueDotted]
		\vyBelowPShort
	;
}}

\def\caseOneUgliestCaseEnd{\myTikZ{
	\caseOneBasicUgliestCase
	\drawYPrimeW
	\draw[redEdge]
		(x) edge (x')
		(x') edge (x'')
		(y') edge (y'')
	;
	\draw[blueEdge]
		(y) edge (y')
		(x') edge (nx')
		(ny') edge (y')
		(w) edge (w')
	;
	\draw[blueDotted]
		\uxPrime
		\sNyPrime
		\xs
		\nxPrimeS
		\vyBelowPShort
	;
}}

\def\basicIsolatedY#1{ 
	\def\xOfT{\lengthP / 3}
	\xWithInterestingNeighbourBasic
	\node[transp] (t) at (#1, 0) {};
	\node[redDot, label={[label distance=\offsetNames]90:$n_y$}] (ny) at (\lengthP - \lengthEdge, 0) {};
	\draw[redDotted]
		(x) edge (ny)
	;
}

\def\basicThreeChildren{
	\def\yOfZ{-\lengthP / 2}
	\basicIsolatedY{\xOfT}
	\yWithInterestingNeighbourBasic
	\nodeS
	\node[redDot, label={[label distance=\offsetNames]90:$z$}] (z) at (\xOfT, \yOfZ) {};
	\node[redDot, label={[label distance=1.2*\offsetNames]0:$n_z$}] (nz) at (\xOfT, \yOfZ + \lengthEdge) {};
	\node[redDot, label={[label distance=1.35*\offsetNames]140:$z'$}] (z') at (\xOfT + \lengthEdge, \yOfZ) {};
	\node[redDot, label={[label distance=1.2*\offsetNames]180:$z''$}] (z'') at (\xOfT + 2*\lengthEdge, \yOfZ) {};
	\node[redDot, label={[label distance=1.5*\offsetNames]180:$n_{z'}$}] (nz') at (\xOfT + \lengthEdge, \yOfZ - \lengthEdge) {};
	\node[transp] (cz1) at ({\xOfT + \lengthEdge + sin(\angleOne)*\lengthEdge}, {\yOfZ - \lengthEdge - cos(\angleOne)*\lengthEdge}) {};
	\node[transp] (cz2) at ({\xOfT + \lengthEdge + sin(\angleTwo)*\lengthEdge}, {\yOfZ - \lengthEdge - cos(\angleTwo)*\lengthEdge}) {};
	\node[transp] (cz3) at ({\xOfT + \lengthEdge + sin(\angleThree)*\lengthEdge}, {\yOfZ - \lengthEdge - cos(\angleThree)*\lengthEdge}) {};
	\node[transpBetweenUndir] (transpNzZPrime) at (\xOfT - 0.5*\lengthEdge, \yOfZ - 0.5*\lengthEdge) {};
	\node[transp] (transpBelowX) at (0, -1.5*\lengthEdge) {};
	\nodeAndHalfEdgeTranspAroundInteresting{\lengthP}{1}{y}
	\draw[redDotted]
		(t) edge (nz)
		(nz') edge (cz1)
		(nz') edge (cz2)
		(nz') edge (cz3)
	;
	\draw[blueEdge]
		(z') edge (z'')
	;
	\draw[blueDotted]
		\xPrimePrimeY
		\nyXPrime
		(ny') edge[bend left, in=140, out=60] (transpBelowX)
		(transpNzZPrime) edge[bend right=45] (z')
		(yAroundInteresting) edge[out=-90, in=20] (z)
	;
	\draw[blueDottedUndir]
		(nz) edge[bend right=45] (transpNzZPrime)
	;
}

\def\threeChildrenStart{\myTikZ{
	\basicThreeChildren
	\edgesInterestingXInTxWithS
	\draw[redEdge]
		(ny) edge (y)
		(y') edge (ny')
		(nz) edge (z)
		(z') edge (nz')
	;
	\draw[blueEdge]
		(y) edge (y')
		(y') edge (y'')
		(z) edge (z')
	;
	\draw[blueDotted]
		(transpBelowX) edge[bend left=30] (s)
		(nz') edge[bend left=50] (transpBelowX)
	;
}}

\def\threeChildrenEnd{\myTikZ{
	\basicThreeChildren
	\draw[redEdge]
		(x) edge (x')
		(x') edge (x'')
		(y) edge (y')
		(y') edge (y'')
		(z) edge (z')
	;
	\draw[blueEdge]
		(x') edge (nx')
		(y) edge (ny)
		(y') edge (ny')
		(z) edge (nz)
		(nz') edge (z')
	;
	\draw[blueDotted]
		\xs
		\nxPrimeS
		\sToTranspBelowX
		\transpBelowXToNzPrime
	;
}}

\def\predBasic{
	\basicIsolatedY{\xOfTPred}
	\nodeY
	\nodeS
	\nodePrime{\lengthP}{1}{y}
	\nodeAndHalfEdgeTranspAroundInteresting{0}{-1}{x}
	\node[redDot, label={[label distance=\offsetNames]90:$z$}] (z) at (\xOfTPred, \yOfZPred) {};
	\draw[redEdge]
		(x) edge (x')
	;
	\draw[blueEdge]
		(y) edge (y')
	;
	\draw[blueDotted]
		(xAroundInteresting) edge[bend right=30] (z)
	;
	\draw[redDotted]
		(t) edge (u)
		(v) edge (z)
	;
}

\def\predOneBasic{
	\def\middlePred{1.5*\lengthEdge}
	\def\yOfU{(\yOfZPred + \middlePred + 2*\lengthEdge) / 2}
	
	\node[redDot, label={[label distance=\offsetNames]180:$u$}] (u) at (\xOfTPred, {\yOfU}) {};
	\node[redDot, label={[label distance=\offsetNames]180:$u'$}] (u') at (\xOfTPred, {\yOfU - \lengthEdge}) {};
	\node[redDot, label={[label distance=\offsetNames]180:$v'$}] (v') at (\xOfTPred, {\yOfZPred - \yOfU + \lengthEdge}) {};
	\node[redDot, label={[label distance=\offsetNames]180:$v$}] (v) at (\xOfTPred, {\yOfZPred - \yOfU}) {};
	
	\predBasic
	\draw[redDotted]
		(u') edge (v')
	;
}

\def\predOneStart{\myTikZ{
	\predOneBasic
	\edgesInterestingXInTxWithS
	\draw[redEdge]
		(ny) edge (y)
		(u) edge (u')
		(v') edge (v)
	;
	\draw[blueEdge]
		(z) edge (y)
	;
	\draw[blueDotted]
		\nyXPrime
		\vz
		\us
	;
}}

\def\predOneEndOne{\myTikZ{
	\predOneBasic
	\draw[redEdge]
		(x') edge (x'')
		(ny) edge (y)
		(v') edge (v)
	;
	\draw[blueEdge]
		(u) edge (u')
		(x') edge (nx')
		(z) edge (y)
	;
	\draw[blueDotted]
		\xs
		\nxPrimeS
		\nyXPrime
		\vz
		\su
	;
}}

\def\predOneEndTwo{\myTikZ{
	\predOneBasic
	\draw[redEdge]
		(u) edge (u')
		(z) edge (y)
		(x') edge (x'')
	;
	\draw[blueEdge]
		(nx') edge (x')
		(ny) edge (y)
		(v) edge (v')
	;
	\draw[blueDotted]
		\xs
		\sNxPrime
		\xPrimeNy
		\zv
		\us
	;
}}

\def\predTwoBasic{
	\def\yOfU{(\yOfZPred + \lengthEdge) / 2}
	\node[redDot, label={[label distance=\offsetNames]180:$u$}] (u) at (\xOfTPred, {\yOfU}) {};
	\node[redDot, label={[label distance=\offsetNames]180:$v$}] (v) at (\xOfTPred, {\yOfZPred - \yOfU}) {};
	\predBasic
}

\def\predTwoStart{\myTikZ{
	\predTwoBasic
	\edgesInterestingXInTxWithS
	\draw[redEdge]
		(ny) edge (y)
		(u) edge (v)
	;
	\draw[blueEdge]
		(z) edge (y)
	;
	\draw[blueDotted]
		\us
		\vz
		\nyXPrime
	;
}}

\def\predTwoEnd{\myTikZ{
	\predTwoBasic
	\draw[redEdge]
		(z) edge (y)
		(x') edge (x'')
	;
	\draw[blueEdge]
		(v) edge (u)
		(ny) edge (y)
		(nx') edge (x')
	;
	\draw[blueDotted]
		\xs
		\sNxPrime
		\us
		\zv
		\xPrimeNy
	;
}}
\def\uglyBasic{
	\caseOneBasicSwapLemmas
	\nodeS
	\draw[blueDotted]
		\xPrimePrimeY
	;
}
\def\uglyOneStart{\myTikZ{
	\uglyBasic
	\edgesTxWithS
	\draw[redEdge]
		(u) edge (u')
		(v) edge (v')
	;
	\draw[blueDotted]
		\vPrimeYPrime
		\vyBelowPShort
	;
}}

\def\uglyOneEnd{\myTikZ{
	\uglyBasic
	\draw[redEdge]
		(x) edge (x')
		(y) edge (y')
		(u) edge (u')
		(x') edge (x'')
	;
	\draw[blueEdge]
		(y') edge (y'')
		(v) edge (v')
		(ny') edge (y')
		(x') edge (nx')
	;
	\draw[blueDotted]
		\vPrimeYPrime
		\yvBelowPShort
		\xs
		\nxPrimeS
		\uxPrime
		\sNyPrime
	;
}}

\def\uglyTwoStart{\myTikZ{
	\uglyBasic
	\edgesTxWithS
	\draw[redEdge]
		(u) edge (u')
		(v) edge (v')
	;
	\draw[blueDotted]
		\vyBelowPShort
	;
}}

\def\uglyTwoEnd{\myTikZ{
	\uglyBasic
	\draw[redEdge]
		(x) edge (x')
		(x') edge (x'')
		(y') edge (y'')
		(v') edge (v)
	;
	\draw[blueEdge]
		(y) edge (y')
		(y') edge (ny')
		(nx') edge (x')
		(u) edge (u')
	;
	\draw[blueDotted]
		\vyBelowPShort
		\xs
		\sNxPrime
		\xPrimeU
		\nyPrimeS
	;
}}

\def\justInterestingNeighbourBasic{
	\interestingNeighbourBasic{0}{-1}{x}
	\node[transp] (tx1) at ({cos(\angleOne)*\lengthEdge}, {sin(\angleOne)*\lengthEdge}) {};
	\node[transp] (tx2) at ({cos(\angleTwo)*\lengthEdge}, {sin(\angleTwo)*\lengthEdge}) {};
	\node[transp] (tx3) at ({cos(\angleThree)*\lengthEdge}, {sin(\angleThree)*\lengthEdge}) {};
	
	\draw[redDotted]
		(x) edge (tx1)
		(x) edge (tx2)
		(x) edge (tx3)
	;
}

\def\interestingNeighbourInTStar{\myTikZ{
	\justInterestingNeighbourBasic
	\draw[redEdge]
		(nx') edge (x')
	;
	\draw[blueEdge]
		(x) edge (x')
		(x') edge (x'')
	;
	\draw[blueDotted]
		\nxPrimeX
	;
}}
\def\interestingNeighbourInTx{\myTikZ{
	\justInterestingNeighbourBasic
	\draw[redEdge]
		(x) edge (x')	
	;
	\draw[blueEdge]
		(nx') edge (x')
		(x') edge (x'')
	;
	\draw[blueDotted]
		\xNxPrime
	;
}}

\def\offsetNamesBig{-1.9em}
\def\angleSwapDef{-60}
\def\xOnABranch#1{#1*cos(\angleSwapDef)*\lengthEdge}
\def\yOnABranch#1{#1*sin(\angleSwapDef)*\lengthEdge}
\def\posOnABranch#1{({\xOnABranch{#1}}, {\yOnABranch{#1}})}

\def\xOnBBranch#1{\xOnABranch{0} + #1*cos(180-\angleSwapDef)*\lengthEdge}
\def\yOnBBranch#1{\yOnABranch{0} + #1*sin(180-\angleSwapDef)*\lengthEdge}
\def\posOnBBranch#1{({\xOnBBranch{#1}}, {\yOnBBranch{#1}})}

\def\xOnCBranch#1{\xOnBBranch{2} + #1*cos(\angleSwapDef)*\lengthEdge}
\def\yOnCBranch#1{\yOnBBranch{2} + #1*sin(\angleSwapDef)*\lengthEdge}
\def\posOnCBranch#1{({\xOnCBranch{#1}}, {\yOnCBranch{#1}})}

\def\swapDefBasic{
	\node[blackDot, label={[label distance=\offsetNames]270:$r$}] (r) at (0, 0) {};
	\node[blackDot] (a1) at \posOnABranch{1} {};
	\node[blackDot, label={[label distance=\offsetNamesBig]0:$v'_1$}] (v'1) at \posOnABranch{2} {};
	\node[blackDot, label={[label distance=\offsetNamesBig]0:$u'_1$}] (u'1) at \posOnABranch{3} {};
	\node[blackDot, label={[label distance=\offsetNamesBig]0:$u_1$}] (u1) at \posOnABranch{4} {};
	\node[blackDot] (a2) at \posOnABranch{5} {};
	\node[blackDot, label={[label distance=\offsetNamesBig]0:$v_1$}] (v1) at \posOnABranch{6} {};
	\node[blackDot] (b1) at \posOnBBranch{1} {};
	\node[blackDot, label={[label distance=\offsetNamesBig]0:$u'_2$}] (u'2) at \posOnBBranch{2} {};
	\node[blackDot, label={[label distance=\offsetNamesBig]0:$u_2$}] (u2) at \posOnBBranch{3} {};
	\node[blackDot, label={[label distance=\offsetNamesBig]0:$v_2$}] (v2) at \posOnBBranch{4} {};
	\node[blackDot] (b2) at \posOnBBranch{5} {};
	\node[blackDot, label={[label distance=\offsetNamesBig]180:$v'_2$}] (v'2) at \posOnCBranch{1} {};
	\node[blackDot] (c1) at \posOnCBranch{2} {};
	\draw[blueEdge]
		(a1) edge (r)
		(v'1) edge (a1)
		(u'1) edge (v'1)
		(b1) edge (r)
		(u'2) edge (b1)
		(b2) edge (v2)
		(v'2) edge (u'2)
		(c1) edge (v'2)
	;
}
\def\swapDefBefore{\myTikZ{
	\swapDefBasic
	\draw[blueEdge]
		(u1) edge (u'1)
		(u2) edge (u'2)
		(a2) edge (u1)
		(v1) edge (a2)
		(v2) edge (u2)
	;
	\draw[redEdge]
		(v1) edge[bend right=40] (v'1)
		(v2) edge (v'2)
	;
}}
\def\swapDefAfter{\myTikZ{
	\swapDefBasic
	\draw[blueEdge]
		(v1) edge[bend right=40] (v'1)
		(v2) edge (v'2)
		(u1) edge (a2)
		(a2) edge (v1)
		(u2) edge (v2)
	;
	\draw[redEdge]
		(u1) edge (u'1)
		(u2) edge (u'2)
	;
}}

\def\xOfR{\lengthKxCaseTwo + 2.5*\lengthEdge}

\def\doubleFigureCentered#1#2#3#4{
	\begin{figure}[!htb]
		\minipage{0.49\textwidth}
		\centering
		#1
		\endminipage\hfill
		\minipage{0.49\textwidth}
		\centering
		#2
		\endminipage
		\caption{#3}
		#4
	\end{figure}
}

\def\fourFiguresCentered#1#2#3#4#5#6{
	\begin{figure}[!htb]
		\minipage{0.49\textwidth}
		\centering
		#1
		\endminipage\hfill
		\minipage{0.49\textwidth}
		\centering
		#2
		\endminipage\\[2em]
		\minipage{0.49\textwidth}
		\centering
		#3
		\endminipage\hfill
		\minipage{0.49\textwidth}
		\centering
		#4
		\endminipage
		\caption{#5}
		#6
	\end{figure}
}

\def\rootSwapBasic{
	\node[redDot, label={[label distance=\offsetNames]90:$x_{i-1}$}] (u) at ({\lengthKxCaseTwo}, 0) {};
	\node[redDot, label={[label distance=\offsetNames]90:$x_i$}] (u') at ({\lengthKxCaseTwo + \lengthEdge}, 0) {};
	\node[redDot, label={[label distance=\offsetNames]90:$r$}] (r) at ({\xOfR}, 0) {};
	\node[transp] (tx1) at ({\xOfR + cos(\angleOne)*\lengthEdge}, {sin(\angleOne)*\lengthEdge}) {};
	\node[transp] (tx2) at ({\xOfR + cos(\angleTwo)*\lengthEdge}, {sin(\angleTwo)*\lengthEdge}) {};
	\node[transp] (tx3) at ({\xOfR + cos(\angleThree)*\lengthEdge}, {sin(\angleThree)*\lengthEdge}) {};
	\draw[redDotted]
		(r) edge (tx1)
		(r) edge (tx2)
		(r) edge (tx3)
		(x) edge (u)
		(u') edge (r)
	;
	\draw[blueDotted]
		(u') edge[bend right=30] (r)
	;
}

\def\rootSwapInterestingBasic{
	\xWithInterestingNeighbourBasic
	\interestingNeighbourBasic{0}{-1}{x}
	\rootSwapBasic
	\draw[blueDotted]
		(x'') edge[bend left=30] (r)
	;
}

\def\rootSwapInterestingTx{\myTikZ{
	\rootSwapInterestingBasic
	\edgesInterestingXInTxWithoutS
	\draw[redEdge]
		(u) edge (u')
	;
}}

\def\rootSwapInterestingTPrime{\myTikZ{
	\rootSwapInterestingBasic
	\draw[redEdge]
		(x) edge (x')
		(x') edge (x'')
	;
	\draw[blueEdge]
		(u) edge (u')
		(nx') edge (x')
	;
	\draw[blueDotted]
		\xNxPrime
		\xPrimeU
	;
}}

\def\rootSwapSmallBasic{
	\node[transp] (tx2) at ({-\lengthEdge - cos(\angleTwo)*\lengthEdge}, {\lengthEdge + sin(\angleTwo)*\lengthEdge}) {};
	\node[redDot, label={[label distance=\offsetNames]90:$x$}] (x) at (0, 0) {};
	\node[redDot, label={[label distance=\offsetNames]0:$x'$}] (x') at (0, \lengthEdge) {};
	\node[redDot] (x'') at (0, 2*\lengthEdge) {};
	\rootSwapBasic
	\draw[redEdge]
		(x') edge (x'')
	;
	\draw[blueDotted]
		(x') edge[bend left=30] (r)
	;
}

\def\rootSwapSmallTStar{\myTikZ{
	\rootSwapSmallBasic
	\draw[blueEdge]
		(x) edge (x')
	;
	\draw[blueDotted]
		(u) edge[bend right=20] (x)
	;
	\draw[redEdge]
		(u) edge (u')
	;
}}
\def\rootSwapSmallTPrime{\myTikZ{
	\rootSwapSmallBasic
	\draw[blueEdge]
		(u) edge (u')		
	;
	\draw[blueDotted]
		(x) edge[bend left=20] (u)
	;
	\draw[redEdge]
		(x) edge (x')
	;
}}

\maketitle

\begin{abstract} 
    The arboricity $\Gamma(G)$ of an undirected graph $G = (V,E)$ is the minimal number $k$ such that $E$ can be partitioned into $k$ forests on $V$. Nash-Williams' formula states that $k = \ceil{ \fracArb(G) }$, where $\fracArb(G)$ is the maximum of $|E_H|/(|V_H| -1)$ over all subgraphs $(V_H, E_H)$ of $G$ with $|V_H| \geq 2$.
    
    The Strong Nine Dragon Tree Conjecture states that if $\fracArb(G) \leq k + \frac{d}{d+k+1}$ for $k, d \in \mathbb N_0$, then there is a partition of the edge set of $G$ into $k+1$ forests on $V$ such that one forest has at most $d$ edges in each connected component.

    Here we prove the Strong Nine Dragon Tree Conjecture when $d \leq 2(k+1)$, which is a new result for all $(k,d)$ such that $d > k+1$. In fact, we prove a stronger theorem. We prove that a weaker sparsity notion, called $(k,d)$-sparseness, suffices to give the decomposition, under the assumption that the graph decomposes into $k+1$ forests. This is a new result for all $(k,d)$ where $d >1$, and improves upon the recent resolution of the Overfull Nine Dragon Tree Theorem for all $(k,d)$ when $d \leq 2(k+1)$.
    As a corollary, we obtain that planar graphs of girth five decompose into a forest and a forest where every component has at most four edges, and by duality, we obtain that $5$-edge-connected planar graphs have a $\frac{4}{5}$-thin tree, improving a result of the authors that $5$-edge-connected planar graphs have a $\frac{5}{6}$-thin tree.
\end{abstract}

\section{Introduction}

Graphs are assumed throughout to have no loops but possibly contain parallel edges. Throughout we use $e(G) = |E(G)|$ and $v(G) = |V(G)|$. This paper deals with \textit{graph decompositions}, which are a partitioning of the edge set of a graph into subgraphs (normally assumed to be disjoint). Graph decompositions become interesting when we add constraints - the typical constraints are restricting the number of subgraphs in the partitioning, and also the types of subgraphs allowed in the partitioning. If we ask for each part to be a specific graph $H$,  then under the mild assumptions that $H$ is connected and has at least three edges, this is an NP-complete problem \cite{Holyer}. To get past the NP-completeness barrier, it is natural to allow the subgraphs in our graph decomposition to come from a family of graphs. One of the simplest families is the class of forests, and thus, it is natural to try to decompose graphs into forests. Of course, every graph can be easily decomposed into forests, so one needs to put a constraint on the number of forests we require. Surprisingly, this question has a complete answer. Nash-Williams' Theorem \cite{nash} states that a graph $G$ has a decomposition into $k$ forests if and only if  $\fracArb(G) \leq k$ where 
$\fracArb(G) = \max_{H \subseteq G, v(H) \geq 2} \frac{e(H)}{v(H) - 1}$. We call $\fracArb$ the \textit{fractional arboricity of $G$}. Thus, the only obstruction to having a decomposition into $k$ forests is the obvious one - having a subgraph that is too dense. This is a very pretty and pleasing result. It almost sounds like the end of the story, we have a complete characterization, and while not immediately clear, we even have a polynomial time algorithm to decide if such a decomposition exists  (see for example, \cite{Edmonds1965}).

However, Xuding Zhu found a rather surprising connection with forest decompositions of graphs where one of the forests has bounded degree, and the game chromatic number. To explain it, we need some definitions. Let $G$ be a graph, and consider the following game played by Alice and Bob. Starting with Alice, the players take turns colouring uncoloured vertices such that at each step, no adjacent vertices have the same colour. Alice wins if the entire graph is coloured, and Bob wins otherwise. The \textit{game chromatic number of $G$} is the minimum integer $k$ such that Alice has a strategy to always win the game. Zhu proved:

\begin{thm}[\cite{gamecolnumber}]
If $G$ decomposes into two forests $T,F$ such that $F$ has maximum degree $d$, then the game chromatic number of $G$ is at most $4+d$. 
\end{thm}

Thus, we have a natural question: When does a graph decompose into $k+1$ forests where one of the forests has bounded degree?  Nash-Williams' Theorem does not handle this situation: all it ensures is that if the fractional arboricity is  at most $k$, we can decompose into $k$ forests. For special graph classes, such as planar graphs, Gonçalves showed in \cite{planar24} that planar graphs can be decomposed into three forests such that one forest has maximum degree at most 4. Therefore one might wonder if it is possible to give a Nash-Williams type theorem for the setting where we decompose into $k+1$ forests, and one of the forests has bounded degree. At first glance, it does not seem possible - if the fractional arboricity is exactly $k$, for integral $k$ we do not have much choice in our selection of forest decomposition (in particular, if $k=1$, there is no choice, and no further strengthening of Nash-Williams Theorem is possible). But one can observe that the fractional arboricity need not be integral - if, for example, the fractional arboricity of a graph $G$ is $k + \varepsilon$ for $k \in \mathbb{N}$ and $\varepsilon >0$ small, then intuitively this means that $G$ decomposes into $k$ forests, plus there is a few extra edges left over, which also form a forest. Thus, one might hope that if $\varepsilon$ is small enough, additional structure can be put onto at least one of the forests. This is the content of the Nine Dragon Tree Theorem\footnote{The name comes from a famous
tree in Kaohsiung, Taiwan, which is far from acyclic.}:

\begin{thm}[Nine Dragon Tree Theorem \cite{ndtt}]	\label{thm:ndtt}
    Let $G$ be a graph and $k$ and $d$ be positive integers. If $\fracArb(G) \leq k + \frac{d}{d + k + 1}$, then there is a decomposition into $k + 1$ forests, where one of the forests has maximum degree at most $d$.
\end{thm}

This theorem was the culmination of a large line of research (for example, \cite{ndtk2},  \cite{Kostochkaetal},  \cite{sndtck1d2}, \cite{Yangmatching}, as a non-exhaustive list) and spawned the now ubiquitous potential method technique \cite{sndtck1d2} (see for example \cite{danandmatt}, \cite{oresconjecture}, \cite{trianglefree}, \cite{NASERASR202281}) before the beautiful proof of Jiang and Yang was found.\\  Note that the choice of bound on the fractional arboricity is best possible:

\begin{thm}[\cite{sndtck1d2}]
\label{tightexamplethm}
    For any positive integers $k$ and $d$ there are arbitrarily large graphs $G$ and a set $S \subseteq E(G)$ of $d+1$ edges such that $\fracArb(G-S) = k + \frac{d}{k+d+1}$ and $G$ does not decompose into $k+1$ forests where one of the forests has maximum degree $d$. 
\end{thm}

Thus again, it appears as if the story is finished, we have a best possible theorem showing that graphs which are sufficiently sparse decompose into $k+1$ forests where one of the forests has maximum degree $d$. It is even the case that the Nine Dragon Tree Theorem, combined with the theorem of Zhu on the game chromatic number, gives best possible bounds on the game chromatic number of high girth planar graphs \cite{CHARPENTIER}, really suggesting that this was the right statement to prove. However, very curiously, the tight examples in Theorem \ref{tightexamplethm} do not rule out the possibility of a strengthening of the Nine Dragon Tree Theorem, in particular, the Strong Nine Dragon Tree Conjecture:

\begin{conj}[Strong Nine Dragon Tree Conjecture \cite{sndtck1d2}]
    Let $G$ be a graph and let $d$ and $k$ be positive integers. If $\fracArb(G) \leq k + \frac{d}{d + k + 1}$, then there is a partition into $k + 1$ forests, where in one forest every connected component has at most $d$ edges.
\end{conj}

The examples did not even rule out the possibility that the last forest was a \textit{star forest} - a forest where every component is isomorphic to a star. Unfortunately, that is false in a rather strong sense. Not only can you not enforce the last component to be a star forest, you can not even obtain one with small diameter, even if one considers \textit{pseudoforest} decompositions instead of forest decompositions. Recall that a pseudoforest is a graph where every connected component contains at most one cycle. More precisely, the following was shown:

\begin{thm}[\cite{mies2023pseudoforest}]
    Let $k, \ell, D \in \mathbb N$, $\varepsilon > 0$, $k \geq 1$ and $\alpha \in \{0, 1\}$.
    There are simple graphs $G$ with\footnote{Note that for $\ell = \floor{\frac{d-1}{k+1}}$ we have $\frac{(k+1)\ell}{(k+1)(\ell + 1)} < \frac{d}{d+k+1}$.}
    $\fracArb(G) < k + \lowerMadBoundWithAlpha + \varepsilon$ that do not have a decomposition into $k+1$ pseudoforests where one of the pseudoforests has maximum degree at most $D$ and the diameter of every component is less than $2\ell + 1 + \alpha$.
\end{thm}

Thus, the Strong Nine Dragon Tree Conjecture seems to be close to the correct level of generality for the structure demanded of the last forest. Unfortunately, it is still wide open. Of course, when $d =1$, the Nine Dragon Tree Theorem when $d=1$, implies the Strong Nine Dragon Tree Theorem, and this was first proven in \cite{Yangmatching}. The $k=1$, $d=2$ case was shown in \cite{Kostochkaetal}. Prior to this paper, the state of the art was:

\begin{thm}[\cite{sndtcDLeqKPlusOne}]
\label{thm:sndtell1}
    The Strong Nine Dragon Tree Conjecture is true when $d \leq k+1$.
\end{thm}

The main contribution of this paper is to prove the Strong Nine Dragon Tree Conjecture when $k+1 < d \leq 2(k+1)$.

\begin{thm}
\label{thm:sndtcDLeq2kPlus2}
The Strong Nine Dragon Tree Conjecture is true when $d \leq 2(k+1)$.
\end{thm}

Note that this solves the conjecture when $\density \in [0, \frac{2}{3})$.
In fact, we prove a stronger theorem which requires some definitions to state. 

\begin{definition}
For positive integers $k$ and $d$, we say a graph $G$ is \textit{$(k,d)$-sparse} if for every subgraph $H$ of $G$ we have
\[\beta(H) := (k+1)(k+d)v(H) - (k+d+1)e(H) -k^{2} \geq 0\]
\end{definition}

This definition may look strange but to build some intuition, compare it to a graph having fractional arboricity at most $k + \frac{d}{k+d+1}$. A graph $G$ has  fractional arboricity at most  $k + \frac{d}{k+d+1}$ if and only if for every subgraph $H$ of $G$ we have 
\[(k+1)(k+d)v(H) - (k+d+1)e(H) -(k+1)(k+d) \geq 0.\]

Thus, being $(k,d)$-sparse is equivalent to relaxing the additive term in the above equations. Of course, this leaves open the possibility that a graph is $(k,d)$-sparse but does not even decompose into $k+1$ forests by Nash-Williams Theorem. This motivates the following definition:

\begin{definition}
Fix a positive integer $k$. For a graph $G$ a subgraph $H$ is $n$-overfull if \[e(H) > n(v(H)-1).\]
\end{definition}

We prove:

\begin{thm}
\label{maintheorem}
Let $k$ and $d$ be positive integers such that $d \leq 2(k+1)$. Every graph $G$ which is $(k,d)$-sparse and has no $(k+1)$-overfull subgraph decomposes into $k+1$ forests such that one of the forests has every component containing at most $d$ edges. 
\end{thm}

When $d \leq 2(k+1)$, this result strengthens the Overfull Nine Dragon Tree Theorem for these values:

\begin{thm}[\cite{Benswebsite}]
Every graph which is $(k,d)$-sparse and has no $(k+1)$-overfull subgraph decomposes into $k+1$ forests such that one of the forests has maximum degree $d$. 
\end{thm}

Naturally we conjecture the following:

\begin{conj}
    Let $k$ and $d$ be positive integers. Every graph which is $(k,d)$-sparse and has no $(k+1)$-overfull subgraph decomposes into $k+1$ forests such that one of the forests has every component containing at most $d$ edges. 
\end{conj}

We pause to point out how the same story has played out for pseudoforests, to hint at possible further extensions. Similarly to Nash-Williams' Theorem, Hakimi's Theorem describes when a graph decomposes into $k$ pseudoforests:

\begin{thm}[\cite{hakimi}]
    A graph $G$ decomposes into $k$ pseudoforests if and only if 
    \[\max_{H \subseteq G} \frac{2e(H)}{v(H)} \leq 2k.\]
\end{thm}

Here the parameter $\max_{H \subseteq G} \frac{2e(H)}{v(H)}$ is called the \textit{maximum average degree} of $G$. One can ask for Nine Dragon Tree type theorems for pseudoforests, and this has resulted in the following three theorems:

\begin{thm}[Pseudoforest Nine Dragon Tree Theorem \cite{ndttPsfs}]
Let $k$ and $d$ be integers. Every graph with maximum average degree at most $2(k + \densNDT)$ decomposes into $k+1$ pseudoforests where one of the pseudoforests has maximum degree $d$. 
\end{thm}

\begin{thm}[Pseudoforest Strong Nine Dragon Tree Theorem \cite{sndtcPsfs}]
    Let $k$ and $d$ be integers. Every graph with maximum average degree at most $2(k + \densNDT)$ decomposes into $k+1$ pseudoforests where one of the pseudoforests has  every component containing at most $d$ edges. 
\end{thm}

\begin{thm} [\cite{mies2023pseudoforest}]
\label{newestpseudo}
Let $k, d \in \mathbb N$, where $k \geq 1$. Let $\ell = \floor{\frac{d-1}{k+1}}$. Then every graph $G$ with maximum average degree at most $2(k + \frac{d}{d+k+1})$ decomposes into  $k + 1$ pseudoforests where one of the pseudoforests $F$ satisfies the following:
\begin{itemize}
\item $F$ is acyclic,
\item every component $K$ of $F$ has $e(K) \leq d$,
\item $diam(K) \leq 2\ell + 2$, and if $\equivModKPlusOne{d}{1}$, then $diam(K) \leq 2\ell+1$,
\item for every component $K$ of $F$ satisfying $e(K) \geq d -z(k-1) +1$, we have $diam(K) \leq 2z$ for any $z \in \mathbb{N}$.
\end{itemize}
\end{thm}

Thus, for pseudoforests, the picture is much clearer. Theorem \ref{newestpseudo} suggests that perhaps one can enforce that the final forest has bounded diameter as a strengthening of the forest version of the Strong Nine Dragon Tree Conjecture. This, plus the fact that there is a matroidal version of the Nine Dragon Tree Theorem \cite{matroidndt}, as well as a digraph version \cite{digraphndt} suggests that we still do not have a great picture of how much structure we can ask in Nine Dragon Tree type theorems. Even more recent papers such as the resolution of the $0$-statement of the Kohayakawa-Kreuter Conjecture \cite{christoph2024resolution} suggest there is a more general theory lurking that we do not quite understand yet.

Now we turn to an application of our theorem. We need a definition first. 

\begin{definition}
    Let $\varepsilon \in (0,1)$ be a real number. Let $G$ be a connected graph, and $T$ a spanning tree of $G$. We say that $T$ is an $\varepsilon$-thin tree if for every cut-set $S \subseteq E(G)$, we have 
    \[ \frac{|E(T) \cap S|}{|S|} \leq \varepsilon.\]
\end{definition}

 In \cite{sndtcDLeqKPlusOne}, we used partial progress towards the Strong Nine Dragon Tree Conjecture when $d=4$ and $k=1$ to show the existence of thin trees in $5$-edge-connected planar graphs.

 \begin{thm}[\cite{sndtcDLeqKPlusOne}]
    Every $5$-edge-connected planar graph admits a $\frac{5}{6}$-thin tree. 
\end{thm}

Following the proof given in \cite{sndtcDLeqKPlusOne} word for word but using the fact that the Strong Nine Dragon Tree Conjecture is true when $d=4$ and $k=1$, which implies that planar graphs of girth at least $5$ decompose into a forest and a forest where every component has at most four edges, we have:

\begin{thm}
    Every $5$-edge-connected planar graph admits a $\frac{4}{5}$-thin tree. 
\end{thm}

Note that $5$-edge-connectivity in the above theorem cannot be replaced with $4$-edge-connectivity even if we replace $\frac{4}{5}$ with any other positive number less than $1$, as shown by Thomassen (see \cite{boundeddiameter} and \cite{Alghasi}). On the other hand, it is likely that $\varepsilon = \frac{4}{5}$ is not optimal, but a new approach would be needed to find thinner trees in $5$-edge-connected planar graphs. On the other hand, with some modifications to the techniques in this paper, perhaps it is possible to prove the following:

\begin{conj}[\cite{boundeddiameter}]
Every planar graph of girth at least $5$ decomposes into a forest and a star forest. 
\end{conj}

This conjecture would imply that every $5$-edge-connected planar graph has two edge-disjoint $\varepsilon$-thin trees, with $\varepsilon = \frac{18}{19}$ \cite{boundeddiameter} (and in fact a refinement gives $\varepsilon = \frac{14}{15}$ \cite{Ramin2018}), which is currently out of reach for our technique. We give some words of motivation for why someone should care about finding thin trees. A famous conjecture of Goddyn\footnote{Goddyn has never published this conjecture, however a discussion of the conjecture can be found in \cite{Gharan2009TheAT}.} states that highly edge-connected graphs always contain thin spanning trees.

\begin{conj}[Goddyn's Thin Tree Conjecture]
For every $\varepsilon > 0$, there exists an integer $c$ depending on $\varepsilon$ such that every $c$-edge-connected graph contains an $\varepsilon$-thin tree.
\end{conj}

There is not much progress towards the Thin Tree Conjecture. For a fixed positive integer $g$, the Thin Tree Conjecture is true for the set of all graphs with bounded orientable genus at most $g$ (where the thinness depends on the integer $g$). It has also been shown that one can find a tree which is ``almost" a thin spanning tree, showing in some sense the conjecture is ``almost" true. We do not even know if the following is true:

\begin{conj}[Goemans, see \cite{Alghasi}]
There exists an $\varepsilon >0$ such that every graph with three edge-disjoint spanning trees contains an $\varepsilon$-thin tree. 
\end{conj}

It is known that a resolution to the Thin Tree Conjecture would imply both a constant factor approximation\footnote{assuming the thin tree can be found in polynomial time} to the asymmetric travelling salesman problem (see \cite{Gharan2009TheAT}), as well as the weak $3$-flow conjecture (again see \cite{Gharan2009TheAT}) (both of which are now known using drastically different techniques - see \cite{ATSP} and \cite{weak3flow}).  We do not delve into the details of these conjectures and problems, as they are not relevant for this paper. We simply note that even though the constant factor approximation algorithm for the asymmetric travelling salesman problem does not use thin trees, when restricted to planar instances, to the authors' knowledge the thin tree approach still gives the best approximation ratio (see \cite{Gharan2009TheAT}). Thus, finding thin trees in planar graphs is still of importance and in particular, optimizing the $\varepsilon$ parameter is a relevant problem. 

We also point out why thin trees and the Strong Nine Dragon Tree Conjecture should be related at all. However, it becomes clearer once we observe that the dual problem to decomposing a graph into $k$ forests is asking when does a graph contain $k$ edge-disjoint spanning trees. Similar to Nash-Williams' Theorem above, Nash-Williams gave the following beautiful characterization:

\begin{thm}[\cite{edgedisjointspanningtrees}]
\label{edgedisjointspanningtrees}
  A graph $G$ contains $k$ edge-disjoint spanning trees if and only if for every partition $\mathcal{P}$ of $V(G)$, the number of edges which have an end in two distinct sets of $\mathcal{P}$ is at least $k(|\mathcal{P}| -1)$.
\end{thm}

Thus, just as the Strong Nine Dragon Tree Conjecture is a refinement of Nash-Williams' Theorem, the Thin Tree Conjecture can be viewed as an extension of Theorem \ref{edgedisjointspanningtrees} by observing that by Theorem \ref{edgedisjointspanningtrees}, every $2k$-edge-connected graph contains $k$ edge-disjoint spanning trees, and thus, by the pigeon-hole principle, for every cut one of the trees is $\frac{1}{k}$-thin with respect to the cut. Thus, the idea is: if the edge-connectivity is significantly larger than $2k$, one should be able to ensure a tree which is thin on all of the cuts.

Now we outline how we prove Theorem \ref{maintheorem} and the new technique in this paper. As a start, we follow the approach that was used and developed in the following papers \cite{ndttPsfs}, \cite{ndtt}, \cite{sndtcDLeqKPlusOne}, \cite{mies2023pseudoforest},  \cite{Yangmatching}. 

The proof proceeds by vertex-minimal counterexample. Step 1 of the proof is to argue that a vertex-minimal counterexample decomposes into $k+1$ forests where $k$ of the forests are spanning trees, and another forest $F$ is left over. We call edges of the spanning trees blue and the forest $F$ red. Effectively this is proven in \cite{ndtt}, and is interesting in its own right as a modification of the proof gives a proof of Nash-Williams' Theorem.

Step 2 of the proof is to pick a decomposition into $k+1$ forests where $k$ are spanning trees and a forest $F$ carefully. Ideally $F$ could be chosen such that all components have size at most $d$ - in which case we would be done immediately. As we cannot, we will attempt to pick such a decomposition that is as ``close" to optimal as possible. We do this in multiple steps. The first step is intuitive, we minimize the number of large components, where we consider many ``small" big components better than one very big component. The second choice is less obvious. As we have to have a component of $F$ that is too large, we pick one and call it $R^*$, and pick a vertex $r$ with largest degree in $R^*$, roughly in the center of the component. Then we orient all edges towards $r$, and consider the induced subgraph of all vertices which are reachable from $r$ by paths which contain possibly both blue and red edges, such that we can select an orientation of the red edges to obtain a directed path. We call this the exploration subgraph. We order the components by size and how ``close" they are to $R^*$ (with $R^*$ being the first component in the ordering), using the directed edges as a way of describing of close. We call this a legal order, and we pick our decomposition such that first the number of large components is small, and second the legal order we find is minimized lexicographically. The idea of the legal order is as follows: if $R^*$ has many components near it small, then one anticipates it will be easy to modify the decomposition to break up $R^*$ without creating any other too big components (or at least smaller too big components). Once we have this, we observe that the exploration graph must have small fractional arboricity by assumption, and this puts some structure on what the components of $F$ look like. In particular, there must be many components of $F$ with few edges (here, and for the remainder of the proof overview few  means at most one edge, and a component is small if it has at most one edge, and big otherwise). This is the content of Section \ref{defcounterexample}.

Step three is to attempt to modify the decomposition. Our first tool comes from \cite{ndtt} and that is the so called special path augmentation. The idea is as follows: imagine I have two components $K$ and $C$ of $F$, and there is a directed edge  between $K$ and $C$, and further $K$ and $C$ are small, so adding the directed blue edge to $F$ does not create a component that is too large. We would love to add this edge to $F$, and find an edge closer to $R^*$ that we can remove, in such a way that either we break up a component too large, or can find a better legal order. It is not at all obvious how to do this, and in general it is quite delicate and complicated. However, the special path argumentation developed in \cite{ndtt} allows us to do precisely this, through possibly many exchanges of edges. As it requires multiple definitions to properly define, and as this idea has now been used in many papers, we simply refer the reader to Section \ref{augmentingspecialpaths} for more details on how the special path augmentation works. The key advantage of this augmentation is that it allows us to say that small components are not near any other small component. We also use them later on for more complicated exchanges, but at a high level one should focus on the fact that the special paths augmentation force no small components to be near other small components. This is the content of Section \ref{augmentingspecialpaths}, and again for experts, this section can be skipped as we add nothing new over the tool developed in \cite{ndtt}.

With this in hand, as mentioned, we just need to show that the exploration graph has few small components.  We just argued in the above paragraph that we cannot have small components close to small components. If we could prove that big components have at most one small component generated by a single spanning tree near them, we would be done. Unfortunately, we will not be able to do this, so we will need to move in steps.

Step 4 is to examine what can happen if we try to exchange edges by a big component near to a small component. A special situation occurs when the big component is $R^*$. Here we will be able to show that we have no small components around it, else we have enough flexibility to find a blue edge which we can exchange with some edge in $R^*$. The reason is again because all blue edges are directed towards $r$, so if I have a blue directed edge from $R^*$ to a small component, walking along the red path from the endpoint of this blue edge to $r$ through the red edges, one of these edges has to be exchangable, and this breaks up $R^*$, and as the component was small, and all new components are smaller than $R^{*}$ (because we picked $r$ to have max degree in $R^{*}$ and roughly in the center of the component).
In general, we do not have the vertex $r$ to use  so we try to say that if we have two small components close to a big component via blue arcs from the same blue spanning tree, then we can exchange some edge in this big component with one of the two blue arcs to a small component to break up the big component or improve the legal order. In \cite{sndtcPsfs} this was shown to always be doable for pseudoforests, but in the case of forests it is not always possible since one of the new components we create might be larger than the old one. However, we will be able to argue that there are at most two small components with blue directed arcs from a single blue spanning tree near any big component. This is proven in Section \ref{sec:threeChildren}, and requires alot of set up. We first consider the structure of how we can possibly have two small components generated by a single blue spanning tree near a big component, and in Section \ref{sec:caseTwo} and Section \ref{sec:caseOne} we argue that if we cannot modify the decomposition when given two small components generated by a single spanning tree near a big component, then we gain a significant amount of structure. Enough structure to argue that there are not three small components near a big component generated by a single blue spanning tree. This is similar to the arguments in \cite{sndtcDLeqKPlusOne}. In fact, we can prove even a little more - we can argue that if two small components are around a big component generated by a single blue tree, then one of these small components has an edge.

At this point, we know we have very few small components, and in the event we have a lot of them around a single component, some of them need to contain at least one edge, and so we are close to concluding the theorem. But the fractional arboricity may still be too low, and naively we cannot modify the structure anymore. As a concrete example, imagine $d=3$ and $k=1$. If a component with three edges (which is considered large), has two children, one which is an edge and one which is an isolated vertex, then if we group these together we have seven vertices, and four edges, given a density of $\frac{4}{7}$, compared to the desired $\frac{3}{5}$. Thus we need to do something to increase the density. Step 5, and the main new idea is to modify big components with certain special big components in a certain way. The high level idea is to show that if we have a big component with two small components as neighbours, then we can find a different big component, which has at most one small component generated by a single blue spanning tree beside it, and this small component contains at least one edge. In this case, this component with its small component is adding more to the fractional arboricity than we need, and it can compensate for the two small components generated by a single blue tree from the one big component. To get this to work is quite delicate. We need to show there is a unique way to pair up components so we don't ``overload" one of the components we using to average out the fractional arboricity, and further that such components even exist. This is extremely complicated, and we do it in tandem with step four in Sections \ref{sec:caseTwo} and \ref{sec:caseOne}. We call these components ``interesting neighbours".

The best way to see how we should do this is to go back to the $d=3$, $k=1$ example. If we really could not somehow modify this component with three edges to break up the component and create a smaller legal order, we have a very specific structure. In particular, we must have a blue directed path from the small component containing an edge, to the big component, and then through the isolated component. Now consider the component right before the big component along this path (naively, its not clear this is not the big component itself - we will argue that is not the case). Then for this new component, the big component will be the interesting neighbour.

Now our goal in Step 5 is to argue that not only does every component have at most two small components generated by a single spanning tree, but actually has at most two ``relevant components" where a relevant component is either an interesting neighbour, or a small component. If we can do that, then the pairing procedure and uniqueness will be easy to prove.

Once we have done all this, Step 6 is to note that if we could not improve our decomposition with the above modifications, by the pairing procedure, we will be able to prove that the fractional arboricity of the exploration graph is too large, a contradiction. This is at this point routine exercise for those comfortable with the arguments in \cite{mies2023pseudoforest,sndtcPsfs,mies2024strong}, and is the content of Section \ref{sec:defineF}.

We structure the paper as follows. In Section \ref{defcounterexample}, we set up our counterexample and explain all the basic definitions. In Section \ref{augmentingspecialpaths}, we review the special paths augmentation from \cite{ndtt} and deduce that small components do not have small child components. In Section \ref{exchangeedgesec} we describe an exchange operation and deduce that the root component of the exploration subgraph has no small or interesting neighbours, which we want to call relevant neighbours. In Sections \ref{sec:caseTwo}, \ref{sec:caseOne} and \ref{sec:threeChildren} we prove some structure of relevant neighbours of components in order to bound the number of these neighbours. In Sections \ref{sec:badIsInteresting} and \ref{sec:defineF} we show how we find components with larger edge density to which we can assign small child components of components that have too many small children.

\section{Defining the counterexample}
\label{defcounterexample}

The goal of this section is to set up everything we need to define  a minimal counterexample to Theorem \ref{maintheorem}. First we pin down some basic notation. For a path $P$ with $k$ vertices, we will write $P = [v_{1},\ldots,v_{k}]$ where $v_{i}v_{i+1}$ is an edge for all $i \in \{1,\ldots,k-1\}$. As we will also consider digraphs, we will use the notation $(u,v)$ to be a directed edge from $u$ to $v$, and we extend the above notation for paths to directed paths if directions are used. For a tree $T$ which has its edges directed towards a root vertex $r$, and $x, y \in V(T)$ such that $x$ is a descendant of $y$ in $T$ we let $\pathIn{T}{x}{y}$ be the unique directed path from $x$ to $y$ in $T$.

For the rest of the paper we fix integers $k,d \in \mathbb{N}$, where $k \geq 1$ and $k + 1 <  d \leq 2(k+1)$ . We always assume that we have a graph $G$ which is a counterexample to Theorem \ref{maintheorem} with minimum number of vertices. Naturally, this implies that $G$ is connected, otherwise each connected component has the desired decomposition, which implies the entire graph has the desired decomposition. Even stronger, we show that $G$ decomposes into $k$ spanning trees and another forest. This fact follows from a minor tweak to the proof of Lemma~2.1 of \cite{ndtt} and the fact that $G$ has no $(k+1)$-overfull set and hence decomposes into $k+1$ forests. We omit the proof.

\begin{lemma}[\cite{ndtt}] \label{lemma:minimalCounterExample}
    Every graph $G$ that is a vertex-minimal counterexample to Theorem \ref{maintheorem} admits a decomposition into forests $T_1, \dots, T_k, F$ such that $T_1, \dots, T_k$ are spanning trees.
\end{lemma}

Note that if $G$ decomposes into $k$ spanning trees and a forest $F$, it follows that $F$ is disconnected. Otherwise, $\fracArb(G) = k+1$, a contradiction.\\
Given a decomposition of $G$, we will want to measure how close it is to satisfying Theorem  \ref{maintheorem}. This is captured in the next definition:

\begin{definition}
    The \textit{residue function $\rho(F)$} of a forest $F$ is defined as the tuple $(\rho_{v(G)-1}(F), \\  \rho_{v(G)-2}(F), \ldots, \rho_{d+1}(F))$, where $\rho_i(F)$ is the number of components of $F$ having $i$ edges.
\end{definition}

We will want to compare residue function values of different forests using lexicographic ordering and are interested in the decomposition with one forest minimizing the residue function. For the next bit of notation, recall that $d \geq 3$ as if $d \leq 2$, then the result follows from Theorem \ref{thm:sndtell1}.

\begin{notation}
    Over all decompositions into $k$ spanning trees and a forest $F$ we choose one where $F$ minimizes $\rho$ with respect to lexicographic order. We call this minimum tuple $\rho^*$. This forest $F$ has a component $R^*$ containing at least $d + 1 \geq 4$ edges.
    We choose a vertex $r \in V(R^*)$ of degree at least $3$ or if this is not possible, we choose $r$ such that there are two edge-disjoint paths in $R^*$  of length at least $2$ starting at $r$. We fix $R^*$ and $r$ for the rest of the paper.
\end{notation}

\begin{definition}
    We define $\mathcal F$ to be the set of decompositions into forests $(T_1, \dots, T_k, F)$ of $G$ such that $T_1, \dots, T_k$ are directed spanning trees of $G$; $R^*$ is a connected component of the undirected forest $F$ and the arcs of $T_1, \ldots T_k$ are directed towards $r$. We let $\mathcal F^* \subseteq \mathcal F$ be the set of decompositions $(T_{1},\ldots,T_{k},F) \in \mathcal F$ such that $\rho(F) = \rho^*$.
\end{definition}

The next definition is simply to make it easier to talk about decompositions in $\mathcal F$.

\begin{definition}
    Let $\mathcal{T} = (T_1, \dots, T_k, F) \in \mathcal F$. We say that the (directed) edges of $T_1, \dots, T_k$ are \textit{blue edges} and the (undirected) edges of $F$ are \textit{red edges}. We define $E(\mathcal{T}) := E(T_1) \cup \dots \cup E(T_k) \cup E(F)$. 
    Furthermore, we let $\redForest(\mathcal T) := F$ and for any $b \inOneToK$ we let $\blueTree(\mathcal T) := T_b$. We also call the connected components of $\redForest(\mathcal T)$ red components. 
    Further, a \textit{blue (directed) path in $\mathcal T$} is a directed path where all edges are from  $\bigcup_{b \inOneToK} E(\blueTree(\mathcal T))$.
\end{definition}

Finally, we can define the critical subgraph which we will focus on for the rest of the paper:

\begin{definition} \label{def:explSubgraph}
    Let $\mathcal{T} \in \mathcal F$.
    The \textit{exploration subgraph} $H_{\mathcal T}$ of $\mathcal{T}$ is the subgraph of the mixed graph $(V, \; E(\mathcal{T}))$ that is induced by the vertex set  consisting of all vertices $v$ for which there is a sequence of vertices $r=x_1, \dots, x_l = v$ such that for all $1 \leq i < l$ it holds: $(x_i, x_{i+1}) \in \bigcup_{b \inOneToK} E(\blueTree(\mathcal T))$ or 
    $x_i x_{i+1} \in E(\redForest(\mathcal T))$.
\end{definition}

We will want to focus on red components of $\mathcal H_{\mathcal T}$ with low edge density:

\begin{definition}
    A red component $K$ is \textit{small} if $e(K) \leq 1$.
\end{definition}

Note that as $d \leq 2(k+1)$, it follows that $K$ is small if and only if $\frac{e(K)}{v(K)} < \density$.\\
Now we turn our focus to the notion of legal orders, which is an ordering of the red components of the exploration subgraph that loosely tells us in what order we should augment the decomposition.

\begin{definition} \label{def:legalOrder}
    Let $\mathcal{T} \in \mathcal F$ and let $\sigma = (R_{1},\ldots,R_{t})$ be a sequence of all red components in $H_{\mathcal T}$. We say $\sigma$ is a \textit{legal order} for $\mathcal T$ if $R_1 = R^*$, and further, for each $1 < j \leq t$, there is an $i_j < j$ such that there is a blue directed edge $(x_j, y_j)$ with $x_j \in V(R_{i_j})$ and $y_j \in V(R_j)$. 
\end{definition}

Observe trivially from the definition of exploration subgraph, a legal order always exists. It will be useful to compare legal orders, and we will again do so using the lexicographic ordering.

\begin{definition} 
    Let $\mathcal T, \mathcal T' \in \mathcal F$ and suppose that $\sigma = (R_1, \dots, R_t)$ and $\sigma' = (R'_1, \dots, R'_{t'})$ are legal orders for $\mathcal T$ and $\mathcal T'$, respectively. We say $\sigma$ is \textit{smaller than} $\sigma'$, denoted $\sigma < \sigma'$ if $(e(R_1), \dots, e(R_t))$ is lexicographically smaller than $(e(R'_1), \dots, e(R'_{t'}))$. If $t \neq t'$, we extend the shorter sequence with zeros  to make the orders comparable.
\end{definition}

To make it easier to discuss legal orders, we introduce some more vocabulary:

\begin{definition}
    Suppose $\sigma = (R_1, \dots, R_t)$ is a legal order for $\mathcal{T} \in \mathcal F$. We write $i_\sigma(v) := j$ for $v \in V(H_{\mathcal T})$ if $v \in V(R_j)$.
    For $U \subseteq V(H_{ \mathcal T})$ we define $i_\sigma(U) := \min \{i_\sigma(v) | v \in  U\}$ (in particular, $i_\sigma(\varnothing) = \infty$) and for any subgraph $H \subseteq H_{\mathcal T}$ we let $i_\sigma(H) := i_\sigma(V(H))$. 
    Moreover, for $\mathcal T, \mathcal T' \in \mathcal F$ let $\Delta(\mathcal T, \mathcal T')$ denote the set of vertices $v$ for which there is a blue arc $(v, u) \in E(\blueTree(\mathcal T))$ for some $b \inOneToK$, but this arc is not contained in $E(\blueTree(\mathcal T'))$.
\end{definition}

 For the purposes of tiebreaking how we pick legal orders, we introduce the next graph:

\begin{definition}
    Let $\mathcal{T} \in \mathcal F$ and let $\sigma = (R_{1},\ldots,R_{t})$ be a legal order for $\mathcal T$. 
    Compliant to Definition \ref{def:legalOrder} we choose a blue arc $(x_j, y_j)$ for all $1 < j \leq t$. There might be multiple possibilities for this, but we simply fix one choice for $\sigma$. 
    By removing all the blue edges from $H_{\mathcal T}$ that are not in $\{(x_j, y_j) \: | \: 1 < j \leq t\}$, we obtain
    the \textit{auxiliary tree of $\mathcal T$ and $\sigma$} and denote it by $Aux(\mathcal T, \sigma)$.
    We always consider $Aux(\mathcal T, \sigma)$ to be rooted at $r$. Furthermore, let $w_j(\sigma) :=  y_j$.\\
    Furthermore, for all $1 < j \leq t$, we say that $R_{i_j}$ is a parent of $R_j$ with respect to $\mathcal T$ and $\sigma$. On the other hand, we call $R_j$ a child of $R_{i_j}$ with respect to $\mathcal T$ and $\sigma$ (that is generated by $(x_j, y_j)$).
\end{definition}

Note that in an auxiliary tree blue arcs are directed away from the root $r$ while in the blue spanning trees of decompositions of $\mathcal F$ they are directed towards $r$.\\
With this, we are in position to define our counterexample. As already outlined, $G$ is a vertex-minimal counterexample to Theorem \ref{maintheorem}. Further, we pick a legal order $\sigma^* = (R^*_1, \ldots, R^*_{t^*})$ for a decomposition $\mathcal T^* \in \mathcal F^*$ such that there is no legal order $\sigma$ with $\sigma < \sigma^*$ for any $\mathcal T \in \mathcal F^*$. We will use these notations for this minimal legal order and decomposition throughout the rest of the paper.

We now outline how we will show that the counterexample graph is not $(k, d)$-sparse.
Let $\mathcal C$ be the set of small red components of $\explSG$ and $\mathcal K$ be the set of red components of $\explSG$ that are not $R^*$ and not small.
In the course of the following sections we will characterize the structure of $\explSG$ and $\sigma^*$ and aim to show that the large density of the components of $\mathcal K$ compensates for the small density of the components of $\mathcal C$. The paper will end with the proof of the following lemma:

\begin{lemma} \label{lemma:densityOfKC}
    There is a function $f: \mathcal C \longrightarrow \mathcal K$ such that for all $K \in \mathcal{K}$, we have:
    \[\frac{e(K) + \sum_{C \in f^{-1}(K)} e(C)}{v(K) + \sum_{C \in f^{-1}(K)} v(C)} \geq \frac{d}{d+k+1}.\]
\end{lemma}

Note that Lemma \ref{lemma:densityOfKC} implies that the fractional arboricity of $G$ is larger than $ k +  d/(d+k+1)$. However, it even implies that $G$ is not $(k, d)$-sparse:

\begin{lemma} \label{lemma:betaLeq0}
    Assuming Lemma \ref{lemma:densityOfKC}, the graph $H := H_{\mathcal{T}^{*}}$ satisfies
    \[\beta(H) = (k+1)(k+d)v(H) - (k+d+1)e(H) - k^{2} < 0,\]
    and thus, $G$ is not $(k,d)$-sparse. Hence, Theorem \ref{maintheorem} is true for $k + 1 < d \leq 2(k+1)$.
\end{lemma}
\begin{proof}
    Suppose to the contrary that $\beta(H) \geq 0$ and let $e_r(H)$ denote the number of red edges of $H$.
    Note that every vertex of $V(H) - r$ has exactly $k$ outgoing blue arcs, and the heads of these arcs are again in $H$ by its definition. $r$ on the other hand does not have any outgoing blue arc. Thus, $e(H) = k(v(H)-1) + e_{r}(H)$ and hence, $\beta(H) = d v(H) + (d+k+1) (k - e_r(H)) - k^2 \geq 0$, which we can rearrange to
    \[
        \frac{d}{d+k+1} \geq \frac{e_r(H) - k + \frac{k^2}{d+k+1}} {v(H)}.
    \]
    Using the function $f$ of Lemma \ref{lemma:densityOfKC} we can rearrange this further to
    \begin{align*}
        \frac{d}{d+k+1}
        &\geq \frac{
            e(R^*) - k(1 - \frac{k}{d+k+1}) + \sum_{K \in \mathcal K \cup \mathcal C} e(K)
        }{
            v(R^*) + \sum_{K \in \mathcal K \cup \mathcal C} v(K)
        }\\
        &\geq \frac{
            d + 1 - k \frac{d + 1}{d+k+1} + \sum_{K \in \mathcal K} \big(e(K) + \sum_{C \in f^{-1}(K)} e(C)\big)
        }{
            d + 2 + \sum_{K \in \mathcal K} \big(v(K) + \sum_{C \in f^{-1}(K)} v(C)\big)
        }\\
    \intertext{By Lemma \ref{lemma:densityOfKC} we have that}
        \frac{d}{d+k+1}
        &\geq \frac{d + 1 - k \frac{d + 1}{d+k+1}} {d + 2}\\
        &= \frac{(d + 1) \frac{d + 1}{d+k+1}} {d+2}\\
        &> \frac{(d+1) \frac{d}{d+1}} {d+k+1}\\
        &= \frac{d}{d+k+1},
    \end{align*}
    which is a contradiction. Thus, $\beta(H) < 0$.
\end{proof}

    Note that Lemma \ref{lemma:densityOfKC} was proven for $d \leq k + 1$ in \cite{sndtcDLeqKPlusOne} where $f^{-1}(K)$ only contains small children of $K$. Thus, the Overfull Strong Nine Dragon Conjecture is true when $d \leq k+1$ by Lemma \ref{lemma:betaLeq0}:

\begin{corollary}
    Theorem \ref{maintheorem} is true for $d \leq k + 1$.
\end{corollary}

In our approach $f^{-1}(K)$ also mainly consists of small children of $K$. However, if $K$  has too many small children, we will find another component of $\mathcal K$  that does not have many small children and that will compensate for the low edge density around $K$. We will define $f$ formally in Section \ref{sec:defineF}.

\section{Augmenting Special Paths}
\label{augmentingspecialpaths}
In this section we consider the first method to find a smaller legal order or shrink a component with more than $d$ edges. This method from \cite{ndtt} roughly works the following way:
if a blue edge $e$ connecting two red components can be coloured red without increasing the residue function, then in certain cases we can find a red edge $e'$ that can be coloured blue in exchange. In order to find this edge we need to look for a certain blue directed path that ends at $e$ and starts at $e'$ and $e'$ has to be closer to $R^*$ with respect to the legal order than $e$ (here we are viewing $e$ as a subgraph with two endpoints to make sense of the term close). First, we formalize the requirements for such a blue path:

\begin{definition}
    Let $\sigma = (R_1, \dots, R_t)$ be a legal order for $\mathcal{T} \in \mathcal F$. We call a blue directed path $P=[v_0, v_1, \dots, v_{l}]$  in $\mathcal T$ that is contained in $H_{\mathcal T}$ \textit{special with respect to $\mathcal T$, $\sigma$ and $(v_{l-1}, v_l)$} if $v_{l-1}$ and $v_l$ are in different components of $\redForest(\mathcal T)$, $i_\sigma(v_l) > i_\sigma(v_0)$ and furthermore, $v_0$ needs to be an ancestor of $v_{l-1}$ in $Aux(\mathcal T, \sigma)$ if both of them are in the same component of $\redForest(\mathcal T)$. \\
    For two special paths $P=[v_0, v_1, \dots, v_{l}]$ and $P'=[v'_0, v'_1, \dots, v'_{l'}]$ with respect to $\mathcal T$, $\sigma$ and $(v_{l-1}, v_l)$ we write $P \leq P'$ if $i_\sigma(v_0) < i_\sigma(v'_0)$, or if $i_\sigma(v_0) = i_\sigma(v'_0)$ and $v_0$ in $Aux(\mathcal T, \sigma)$ is an ancestor  (with respect to the root $r$) of $v'_0$. We call a special path $P$ with respect to $\mathcal T$, $\sigma$ and $(x, y)$ \textit{minimal} if there is no special path $P' \neq P$ with respect to $\mathcal T$, $\sigma$ and $(x, y)$ with $P' \leq P$.
\end{definition}
Note that for every special path $P'$ with respect to $\mathcal T$, $\sigma$ and $(x, y)$ there exists a minimal special path $P$ with respect to $\mathcal T$, $\sigma$ and $(x, y)$ such that $P \leq P'$. Furthermore, note that if we have a minimal special path $P=[v_0, v_1, \dots, v_l]$ with respect to $\mathcal T$, $\sigma$ and $(v_{l-1}, v_l)$, we have $v_0 \neq r$ because $r$ has no outgoing blue edge by construction. Therefore, $v_0$ has a parent vertex in $Aux(\mathcal T, \sigma)$, which we denote by $v_{-1}$. Note that the edge $v_{-1} v_0$ is red because of the minimality of $P$, and since all blue edges in $Aux(\mathcal T, \sigma)$ are directed away from $r$ in the auxiliary tree.\\
The following lemma describes which modifications to the decomposition can be made if a minimal special path exists and how they change the legal order.

\begin{lemma}[cf.\ Lemma 2.4 and Corollary 2.5 in \cite{ndtt}]	\label{lemma:specialPaths}
    Let $\sigma = (R_1, \dots, R_t)$ be a legal order for $\mathcal T = (T_1, \ldots, T_k, F) \in \mathcal F$.\\
    Furthermore, let $P = [v_0, v_1, \dots, v_l]$ be a minimal special path with respect to $\mathcal T$, $\sigma$ and $(v_{l-1},v_{l})$, let $i_0 := i_\sigma(v_0)$.\\
    Then there is a partition into forests $\mathcal{T}' = (T'_1, \dots, T'_k, F')$ of $G$ such that $T'_1, \dots, T'_k$ are spanning trees rooted at $r$ whose edges are directed to the respective parent vertex, the forest $F'$ exclusively consists of undirected edges and if $R'$ is the component containing $r$ in $F'$, we have that:
    \begin{enumerate}
        \item $F' = \big( F + v_{l-1} v_l \big) - v_{-1}v_0$.
        \item $(v_0, v_{-1}) \in \bigcup_{b=1}^k E(T'_b)$.
        \item $\big\{(x, y) \in E(T'_b) \: | \: i_\sigma(x) < i_0\big\} 
        = \big\{(x, y) \in E(T_b) \: | \: i_\sigma(x) < i_0\big\}$ for all $b \inOneToK$.
        \item $\begin{aligned}[t]
            \bigcup_{b=1}^k \big\{uv \: | \: (u, v) \in E(T'_b)\big\} 
            = \bigg(\Big(\bigcup_{b=1}^k \big\{uv \: | \: (u, v) \in E(T_b) \big\}\Big) - v_{l-1}v_l \bigg) + v_0v_{-1}.
        \end{aligned}$
        \item If $i_0 > 1$, then $R^* = R'$, $\mathcal{T}' \in \mathcal F$, $i_{\sigma^*}(\Delta(\mathcal T, \mathcal T')) = i_0$ and there exists a legal order $\sigma' = (R'_1, \dots, R'_{t'})$ for $\mathcal{T}'$ with $R'_j = R_j$ for all $j < i_0$ and $e(R'_{i_0}) < e(R_{i_0})$, where $R'_{i_0}$ is the component of $v_{-1}$ in $F'$. Thus, $\sigma' < \sigma$.
    \end{enumerate}
\end{lemma}

Since we want to find smaller orders than $\sigma^*$ in our proofs to arrive at a contradiction, it is desirable that the fifth point holds in an application of Lemma~\ref{lemma:specialPaths}. In the event this is not the case, we can still gain more structure. We want to show this more formally with the next lemma:

\begin{lemma}		\label{lemma:v0NotInR}
    If Lemma \ref{lemma:specialPaths} is applicable such that $\rho(F + v_{l-1} v_l) = \rho^*$ holds, then $e(R_{i_0}) \leq d$ and therefore $i_0 > 1$ such that 5.\ can be applied. Moreover, in this case we have $\mathcal{T}' \in \mathcal F^*$ for the partition obtained.
\end{lemma}
\begin{proof}
    We use the notation of Lemma~\ref{lemma:specialPaths}. We have that $R_{i_0}$ is split into two strictly smaller components in $F'$. If $\rho(F + xy) = \rho^*$ and $e(R_{i_0}) > d$, then we had $\rho(F') < \rho^*$, a contradiction.
\end{proof}

\begin{corollary} \label{cor:specialPathInStandardCase}
    Let $\sigma = (R_1, \ldots, R_t)$ be a legal order for $\mathcal T \in \mathcal F$.
    If there is a special path $P = [v_0, \ldots, v_l]$ with respect to $\mathcal T$, $\sigma$ and $(v_{l-1}, v_l)$ such that $e(R_i) \leq e(R^*_i)$ for all $i \in \{1, \ldots, i_\sigma(v_0)\}$, then $\rho(\redForest(\mathcal T) + v_{l-1} v_l) > \rho^*$.
\end{corollary}

Corollary \ref{cor:specialPathInStandardCase} shows how we will make use of special paths later. After doing some exchanges between the forests of $\mathcal T^*$ we will obtain a decomposition $\mathcal T$ and a special path as described in the corollary, but it will also hold $\rho(\redForest(\mathcal T) + v_{l-1} v_l) = \rho^*$.
Next we want to look at what consequences Lemma \ref{lemma:specialPaths} has on the relation of children and parents:

\begin{corollary}[Corollary 2.5 from \cite{ndtt}] \label{cor:sumOfChildRelation} \phantom{bla} \\ 
    Let $C$ be a child of $K$ with respect to $\mathcal T^*$ and $\sigma^*$ that is generated by $(x, y)$. Then $e(K) + e(C) \geq d$. 
\end{corollary}
\begin{proof}
    Suppose to the contrary that $e(K) + e(C) < d$. Since $i_{\sigma^*}(y) > i_{\sigma^*}(x)$, we have that $[x, y]$ is a special path with respect to $\mathcal T^*$, $\sigma^*$ and $(x, y)$. Furthermore, $\rho(\redForest(\mathcal T^*) + xy) = \rho^*$, but this contradicts Corollary \ref{cor:specialPathInStandardCase}.
\end{proof}

\section{Exchanging Edges}
\label{exchangeedgesec}
In this section, we define a useful exchange operation and show how to reorient edges after the exchange to maintain the proper structure of the decomposition. We then use this exchange operation to show that $R^*$ does not have small children. After that we prove a lemma which shows two useful cases which can occur when trying to exchange edges. For this section we define $\mathcal T \in \mathcal F$ together with a legal order $\sigma$ of $\mathcal T$.

\begin{definition}
    Let $e \in E(\blueTree(\mathcal T))$ for some $b \inOneToK$ and $e' \in E(\redForest(\mathcal T))$. If $(\blueTree(\mathcal T) - e) + e'$ is a spanning tree and $(F - e') + e$ is a forest (ignoring orientations), we say that $e'$ can be exchanged with $e$, and say that $e \leftrightarrow e'$ holds in $\mathcal T$.
\end{definition}

The next lemma explains when $e \leftrightarrow e'$ holds and how the two edges can be exchanged in order to obtain a new decomposition in $\mathcal F$. We omit the proof since it is easy. However, we refer the reader to Figure \ref{fig:swapDef} for an illustration.

\begin{lemma} \label{lemma:exchange}
    Let $u \in V(G) - r$, $u'$ be the parent vertex of $u$ in $\blueTree(\mathcal T)$ for some $b \inOneToK$ and $e = vv' \in E(\redForest(\mathcal T))$ such that $(\redForest(\mathcal T) + uu') - e$ does not contain a cycle.\\
    Then, the following are equivalent:
    \begin{enumerate}[(a)]
        \item $(u, u') \leftrightarrow e$ holds in $\mathcal T$.
        \item The edge $(u,u')$ lies in the unique cycle (ignoring orientations) of $\blueTree(\mathcal T) + e$.
        \item Up to relabelling $v$ as $v'$, $v$ is a descendant of $u$ in $\blueTree(\mathcal T)$ and $v'$ is not.
    \end{enumerate}
    Furthermore, if these conditions are met, then after exchanging $(u, u')$ and $e$ between $\blueTree(\mathcal T)$ and $\redForest(\mathcal T)$, orienting $e$ towards $v'$, removing the orientation of $e$ and reorienting the path $\pathInBlueTree{\mathcal T}{v}{u}$, the resulting decomposition is again in $\mathcal F$, and we say we obtain the resulting decomposition from $\mathcal T$ by performing $(u, u') \leftrightarrow e$.
\end{lemma}

\doubleFigureCentered{
    \swapDefBefore
}{
    \swapDefAfter
}{
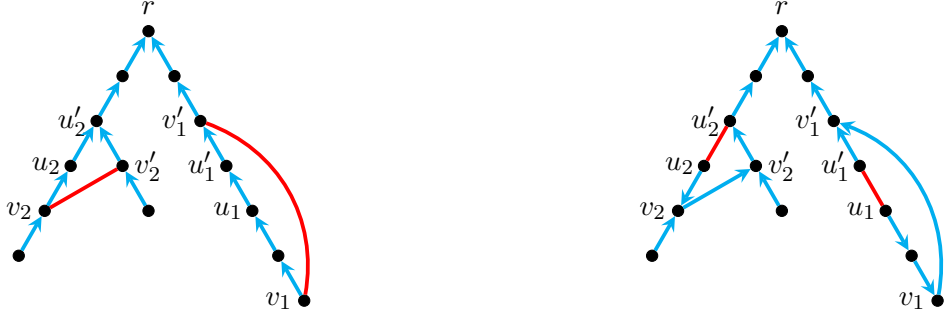
    An example where we have $(u_1, u'_1) \leftrightarrow v_1 v'_1$ and $(u_2, u'_2) \leftrightarrow v_2 v'_2$.
}{
    \label{fig:swapDef}
}

\begin{notation}
    Let $K$ be a red component of $\mathcal T^*$ and let $x \in V(K)$. If $x$ has degree $1$ in $K$, then let $n_x$ be its only neighbour in $K$. If $(x, x') \in E(\blueTree(\mathcal T^*))$ and $x' \neq r$, then let $x''$ denote the parent of $x'$ in $\blueTree(\mathcal T^*)$. 
    If there are two arcs from $x$ to $x'$ (in different blue spanning trees), then $x''$ denotes the parent of the tree whose arc $(x, x')$ we are considering in the respective context.
\end{notation}

We will now define what an interesting neighbour component of $K$ is. If $K$ has such a neighbour, then in certain situations we are able to exchange edges and improve the legal order. Note that the situation in the following definition is depicted on the left side of Figure \ref{fig:interestingNeighbour}.

\begin{definition}
    Let $K$ be a red component of $\explSG$ and $x \in V(K)$. Let $b \inOneToK$ such that $(x, x') \in E(\blueTree(\mathcal T^*))$, where $x' \neq r$ has exactly one incident red edge $x' n_{x'}$ within its red component $L \neq K$ in $\explSG$. Furthermore, let the red component in $\explSG$ of $x''$ not contain an edge and be a child of $L$ with respect to $\mathcal T^*$ and $\sigma^*$. Moreover,  suppose there is a directed path from $n_{x'}$ to $x$ in $\blueTree(\mathcal T^*)$.\\
    Then we say \textit{$L$ is an interesting neighbour of $K$ generated by $(x, x')$} (or generated by $\blueTree(\mathcal T^*)$ ).\\
    Furthermore, we call a red component a \textit{relevant neighbour of $K$} if it is a small child of $K$ with respect to $\mathcal T^*$ and $\sigma^*$ or if it is an interesting neighbour of $K$.
\end{definition}

\doubleFigure{
    \center
    \interestingNeighbourInTStar
}{
    \center
    \interestingNeighbourInTx
}{
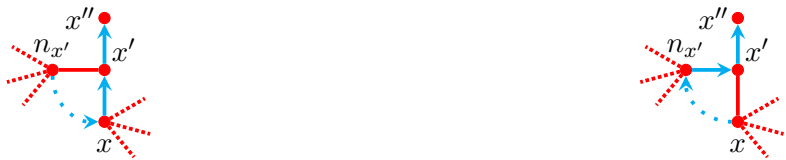
    An interesting neighbour generated by $(x, x')$ in $\mathcal T^*$ and in $\mathcal T_x$.
}{
    \label{fig:interestingNeighbour}
}

\begin{obs} \label{obs:interestingHasGeqDEdges}
    Let $C$ be an interesting neighbour of $K$ generated by $(x, x')$. Then $e(C) \geq d$ by Corollary \ref{cor:sumOfChildRelation} (as $x''$ has no incident red edges) and thus, a relevant neighbour cannot be both interesting and small.
\end{obs}

We introduce more notation that will be useful when performing exchanges near interesting components.

\begin{notation}
Let $C_x$ be a relevant neighbour of $K$ generated by $(x, x')$. If $C_x$ is interesting, we have that $n_{x'}$ is a descendant of $x$ in $\blueTree(\mathcal T^*)$ and $x'$ is not. Let $\mathcal T_x$ denote the decomposition we obtain from $\mathcal T^*$ by performing $(x, x') \leftrightarrow n_{x'} x'$. If $C_x$ is small, then we let $\mathcal T_x := \mathcal T^*$.
Furthermore, we let $(\bar x, \bar x') := (x, x')$ if $C_x$ is small, and $(\bar x, \bar x') := (x', x'')$ if $C_x$ is interesting.
Moreover, we let $c_x = 0$ if $e(C_x) = 0$, and $c_x = 1$ otherwise. If $C_x$ is interesting, then let $C'_x := C_x - x'$.
\end{notation}

Note that $\mathcal T_x$ is depicted on the right side of Figure \ref{fig:interestingNeighbour} and furthermore, note that $(\bar x, \bar x')$ always generates a small child.\\
The procedure described in Lemma \ref{lemma:exchange} enables us to enforce that $R^*$ does not have relevant neighbours, and thus, the density around $R^*$ is high:

\begin{lemma} \label{lemma:rootNoChildren}
    The component $R^*$ does not have relevant neighbours.
\end{lemma}
\begin{proof}
    Suppose to the contrary that $R^*$ has a relevant neighbour $C_x$ generated by $(x, x') \in E(\blueTree(\mathcal T^*))$ for some $b \inOneToK$.
     Let $P = [x_1, \ldots, x_n]$ be the path from $x$ to $r$ in $\redForest(\mathcal T_x)$. Let $i \in \{1, \ldots, n\}$ such that $x_i$ is the first vertex on $P$ that is not a descendant of $\bar x$ in $\blueTree(\mathcal T_x)$. This vertex exists and $i > 1$ since $x$ is a descendant of $\bar x$ in $\blueTree(\mathcal T_x)$ and $r$ is not. We obtain $\mathcal T'$ from $\mathcal T_x$ by performing $(\bar x, \bar x') \leftrightarrow x_{i-1} x_i$, which is depicted in Figure \ref{fig:rootSwap}.
    The component $K_r$ of $r$ in $\redForest(\mathcal T')$ contains at least two edges by the way we chose $r$. For the component $K_x$ of $x$ in $\redForest(\mathcal T')$ we have 
    \[
        e(K_x) 
            \leq e(R^*) - |\{x_{i-1} x_i\}| - e(K_r) + |\{xx'\}| + c_x 
            < e(R^*).
    \]
    Since $K_r$ is a proper subgraph of $R^*$, we obtain a contradiction to the minimality of $\rho^*$.
\end{proof}

\fourFiguresCentered{
    \rootSwapSmallTStar
}{
    \rootSwapSmallTPrime
}{
    \rootSwapInterestingTx
}{
    \rootSwapInterestingTPrime
}{
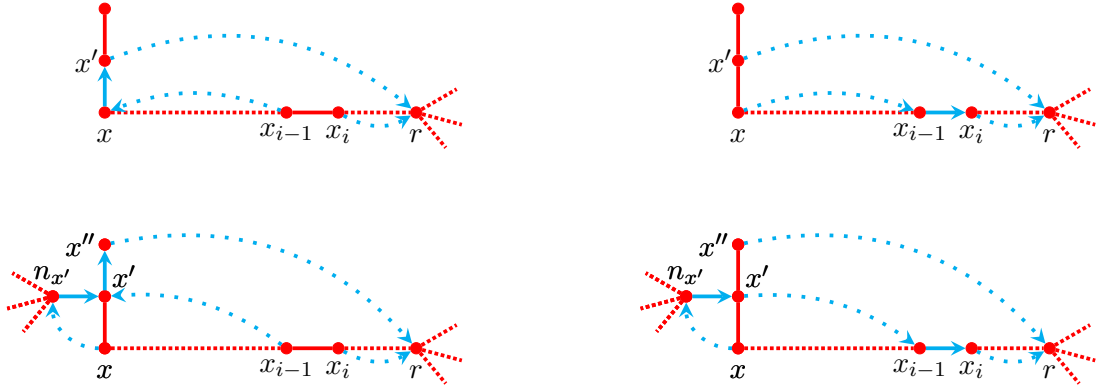
    The decomposition $\mathcal T_x$ and $\mathcal T'$ in the proof of Lemma \ref{lemma:rootNoChildren}: in the first row we have the case where $R^*$ has a small child and in the second row $R^*$ has an interesting neighbour.
}{
    \label{fig:rootSwap}
}

Note that an optimal legal order in the decomposition $\mathcal T'$ in the proof of Lemma \ref{lemma:rootNoChildren} might look completely different than $\sigma^*$ since $\bar x \bar x'$ is red in $\mathcal T'$ and $\pathInBlueTree{\mathcal T_x}{x_{i-1}}{\bar x}$ is reoriented in $\blueTree(\mathcal T')$.
If we perform similar swaps later on that do not guarantee to decrease the residue function, this can be a huge problem, especially if $\pathIn{\mathcal T_x}{x_{i-1}}{\bar x}$ contains blue arcs of $Aux(\mathcal T^*, \sigma^*)$. In this case, we might have increased the legal order drastically. The following technical lemma shows how we can work around this problem: first, we obtain a decomposition $\barT$ by performing some exchanges of edges. We choose these exchanges such that $\sigma^*$ is still intact in $\barT$ up to an index $i$. In $\barT$ we hope to find a special path starting at a vertex $v_0$ with $i_{\sigma^*}(v_0) \leq i$. By augmenting $\barT$ with the help of this special path we obtain a smaller legal order than $\sigma^*$. This is a contradiction if the resulting decomposition also has an optimal residue function value. This idea and the exact conditions are formalized in the following lemma, which will be used by most of the exchange lemmas to come.

\begin{lemma} \label{lemma:newSmallerByPath}
    Let $\barT$ be a decomposition containing a blue arc $(a, a')$.
    Let $A$ be the set of vertices which have a blue directed path to $a$ in $\barT$ and let $B$ be the set of all the vertices of $\Delta(\mathcal T^*, \barT)$ which are not in $A$.
    Let $i_A := i_{\sigma^*}(A)$ and let $L$ be the component of $w_{i_{A}}(\sigma^*)$ in $\redForest(\barT)$.
    We have that $\rho(\redForest(\barT) + aa') > \rho^*$  if the following conditions are met:
    \begin{enumerate}[a)]
        \item $i_{\sigma^*}(a') > i_A$.
        \item $i_{\sigma^*}(B) \geq i_A$.
        \item If $L \neq R^*_{i_A}$ and $e(L) = e(R^*_{i_A})$, then $L$ contains a vertex of $A$.
        \item If the component of a vertex $v$ of $V(R^*_{i_A}) \cap V(L)$ in $\redForest(\barT)$ contains more than $e(R^*_{i_A})$ edges, then $v \in A$.
    \end{enumerate}
\end{lemma}
\begin{proof}
    Suppose to the contrary that $\rho(\redForest(\barT) + aa') = \rho^*$ (recall that by definition of $\rho^{*}$, $\rho(\redForest(\barT) + aa') < \rho^*$ never occurs).
    By b) there is a legal order $\bar\sigma = (\bar R_1, \ldots, \bar R_{\bar t})$ of $\barT$ with $\bar R_i = R^*_i$ for all $i < i_A$.
    We choose $\bar R_{i_A} := L$. This is possible since the arc $(u, w_{i_{A}}(\sigma^*))$ of $Aux(\mathcal T^*, \sigma^*)$ is also present in $T_{\bar \sigma}$, or otherwise we would have $u \in A \cup B$  
    although $i_{\sigma^*}(u) < i_A$. 
    Furthermore, we have that  $w_{i_{A}}(\sigma^*) \notin A$, as otherwise $u \in A$. Hence, $e(L) \leq e(R^*_{i_A})$ by d).
    In fact, $e(L) = e(R^*_{i_A})$ as otherwise we obtain $\bar\sigma < \sigma^*$ if we complete $\bar \sigma$ arbitrarily, contradicting our choice of $\sigma^{*}$.
    \\
    Note that by the definition of $i_A$ there is a vertex of $A$ in $V(R^*_{i_A})$. Considering c) we have that $L$ always contains a vertex of $A$.
    Thus, we can complete $\bar\sigma$ after $L$ by traversing a blue path from a vertex of $V(A) \cap V(L)$ to $a'$ going over $a$ in $\barT$ and adding every red component to $\bar\sigma$ that is not yet in $\bar\sigma$. We obtain $i_{\bar\sigma}(a') > i_A$ by a). Thus, there is a minimal special path starting at $L$ with respect to $\barT$, $\bar \sigma$ and $(a, a')$, which is a contradiction to Corollary \ref{cor:specialPathInStandardCase}.
\end{proof}

We can now generalize Corollary \ref{cor:sumOfChildRelation} by considering the case where $C$ is an interesting neighbour of $K$.

\begin{lemma} \label{lemma:sizeOfKIfHasRelevantChild}
    Let $K$ be a red component of $\explSG$ and $C_x$ a relevant neighbour of $K$. Then $e(K) \geq d - c_x$ and in particular, $K$ is not small.
\end{lemma}
\begin{proof}
    By Corollary \ref{cor:sumOfChildRelation} we know that if the lemma is not true, then $e(K) \leq d - 2$ and $C_x$ is an interesting neighbour of $K$. Let $C_x$ be generated by $(x, x') \in E(\blueTree(\mathcal T^*))$ for some $b \inOneToK$.
    We want to find a contradiction to Lemma \ref{lemma:newSmallerByPath} and choose $\barT := \mathcal T_x$ and $(a, a') := (x', x'')$. Note that $B = \varnothing$  since all vertices in $\Delta(\mathcal T^{*}, \barT)$ belong to the path $\pathInBlueTree{\mathcal T^*}{n_{x'}}{x}$, but this path is reoriented in $\blueTree(\mathcal T_x)$ and $\blueTree(\mathcal T_x)$ also contains the arc $(n_{x'},x')$, and thus, all of these vertices have a blue path to $x'$ in $\blueTree(\mathcal T_x)$, and so cannot be in $B$. Thus, b) holds and it is also clear that a) holds.
    Let $K'$ be the component of $x$ in $\redForest(\mathcal T_x) + x'x''$ and note that $C'_x$ is also a component of this forest. We have $e(C'_x) = e(C_x) - 1$ and $e(K') = e(K) + |\{xx', x'x''\}| \leq d$. Thus, d) holds and $\rho(\redForest(\mathcal T_x) + x'x'') = \rho^*$.  As $x' \in A$, we also have that c) holds contradicting Lemma \ref{lemma:newSmallerByPath}.
\end{proof}

Our goal is to bound the number of relevant neighbours of red components. To achieve this, we are looking for exchanges improving the legal order if a component has two relevant neighbours generated by edges $(x, x'), (y, y')$ of the same tree. The next lemma shows that we can find an edge $e$ on the red path from $x$ to $y$ that can be exchanged with either $(\bar x, \bar x')$ or $(\bar y, \bar y')$. It also distinguishes between two cases depending on properties of the red path between $x$ and $y$. These cases can be seen in Figure \ref{fig:cases}.

\begin{lemma}		\label{lemma:cases}
    Let $b \inOneToK$, let $K \neq R^*$ be a red component of $\explSG$ and let $x$ and $y$ be vertices of $K$.
    Furthermore, let $C_x$ be a relevant neighbour of $K$ generated by $(x, x') \in E(\blueTree(\mathcal T^*))$ and let $y$ not be a descendant of $\bar x$.
    Let $y' \neq x'$ be the parent vertex of $y$ in $\blueTree(\mathcal T^*)$.
    Furthermore, let $[x_1, \dots, x_n]$ be the path from $x$ to $y$ in $\redForest(\mathcal T^*)$.
    Then one of the following two cases applies:
    \begin{itemize}
        \item There is an $i \in \{1, \ldots, n-1\}$ such that $x_i$ is a descendant of $x$ in $\blueTree(\mathcal T_x)$ and $x_{i+1}$ is not. Thus, $(\bar x, \bar x') \leftrightarrow x_i x_{i+1}$ holds in $\mathcal T_x$. Furthermore, $x_{i+1}$ is a descendant of $y$ 
        in $\blueTree(\mathcal T_x)$.\\
        In this case we say that $x \caseTwo y$ holds with edge $(x_i, x_{i+1})$.
        \item It does not hold $x \caseTwo y$. There is an $i \in \{1, \ldots, n-2\}$ such that $x_i$ is a descendant of $x$ in $\blueTree(\mathcal T_x)$ and $x_{i+1}$ is not and $i$ is maximal with this property. Thus, $(\bar x, \bar x') \leftrightarrow x_i x_{i+1}$ holds in $\mathcal T_x$.
        Furthermore, there is a $j \in \{i+2, \ldots, n\}$ such that $x_j$ is a descendant of $y$ in $\blueTree(\mathcal T_x)$ and $x_{j-1}$ is not and $j$ is minimal with this property. Thus, we have $(y, y') \leftrightarrow x_j x_{j-1}$ in $\mathcal T_x$.
        In this case we say that $x \caseOne y$ holds with edges $(x_i, x_{i+1}), (x_{j-1}, x_j)$.
    \end{itemize}
\end{lemma}
\begin{proof}
    Suppose it does not hold $x \caseTwo y$. Note that $y$ also does not have a path to $\bar x$ in $\blueTree(\mathcal T_x)$: if it had one, then it would have a path in $\blueTree(\mathcal T^*)$ to a vertex of $\Delta(\mathcal T^*, \mathcal T_x) = \pathInBlueTree{\mathcal T^*}{n_{x'}}{x}$ and thus, would be a descendant of $x$ in $\blueTree(\mathcal T^*)$, a contradiction. Thus,  there is an $i \in \{1,\ldots,n-2\}$ as described in $\caseOne$ since $x$ is a descendant of $\bar x$ in $\blueTree(\mathcal T_x)$ and $y$ is not. To see that $i \neq n-1$, if it were, we would have $x \caseTwo y$ with edge $(x_i, x_{i+1})$. Now, let $j \geq i + 2$ be minimal such that $x_j$ is a descendant of $y$ in $\blueTree(\mathcal T^*)$. Observe that $j$ exists since $y$ is a descendant of $y$ in $\blueTree(\mathcal T^*)$. If $x_{j-1}$ was a descendant of $y$, then $j = i+2$ and thus, we have $x \caseTwo y$, a contradiction.
\end{proof}

\singleFigure{
    \caseTwoCaseDef
    \caseOneCaseDef
}{
    We have $x \caseTwo y$ on the left and $x \caseOne y$ in $\mathcal T_x$ on the right.
}{
    \label{fig:cases}
    }

We end this section with an observation for the case where the reoriented blue path of an exchange is trivial.

\begin{lemma} \label{lemma:neighOfXIsDescendant}
    Let $K$ be a red component of $\explSG$ and $C_x$ a relevant neighbour of $K$ generated by $(x, x') \in E(\blueTree(\mathcal T^*))$ for some $b \inOneToK$. Furthermore, let $xu \in E(K)$ be a red edge and let $K_x$ be the component of $K - xu$ containing $x$.\\
    If $e(K_x) = 0$ and there is no path in $\blueTree(\mathcal T_x)$ from $u$ to $\bar x$ (and thus, no path from $u$ to $x$), then $d = 3$, $e(K) = 2$ and $c_x = 1$.
\end{lemma}
\begin{proof}
    Suppose that $e(K_x) = 0$ and there is no path in $\blueTree(\mathcal T_x)$ from $u$ to $\bar x$.
    We obtain $\mathcal T'$ from $\mathcal T_x$ by performing $(\bar x, \bar x') \leftrightarrow xu$.
    Let $K'_x$ and $K'_u$ be the components in $\redForest(\mathcal T')$ of $x$ and $u$, respectively.
    Note that if $C_x$ is small, we have $\blueTree(\mathcal T') = (\blueTree(\mathcal T^*) - (x, x')) + (x, u)$ and otherwise $\blueTree(\mathcal T') = \blueTree(\mathcal T^*) \setminus \{(x, x'), (x', x'')\} \cup \{(x, u), (x', n_{x'})\}$ since $\pathInBlueTree{\mathcal T^*}{n_{x'}}{x}$ is reoriented when obtaining $\mathcal T_x$ from $\mathcal T^*$, but the resulting path is again reoriented when obtaining $\mathcal T'$ from $\mathcal T_x$. Hence, there is a legal order $\sigma' = (R'_1, \dots, R'_{t'})$ for $\mathcal T'$ with $R'_j = R^*_j$ for all $j < i := \min \{i_{\sigma^*}(K), i_{\sigma^*}(C_x)\}$ and $(R^*_i, R'_i) \in \{(K, K'_x), (K, K'_u), (C_x, C'_x), (C_x, K'_x)\}$ (where $R'_i = C'_x$ or $R^*_i = C_x$ can only be possible if $C_x$ is interesting).
    Note that $e(K'_u) < e(K)$ and if $C_x$ is interesting, we have $e(C'_x) < e(C_x)$ and $e(C_x) \geq d$ by Observation \ref{obs:interestingHasGeqDEdges}. Since $e(K'_x) = 1 + c_x  < d$, we have $\rho(\redForest(\mathcal T')) = \rho^*$. Thus, it must be $R^*_i = K$, $R'_i = K'_x$ and $e(K'_x) \geq e(K)$ or otherwise we have $\sigma' < \sigma^*$, which is a contradiction. 
    Thus, $e(K) \leq 2$ and it cannot be $e(K) = 1$ because of Lemma \ref{lemma:sizeOfKIfHasRelevantChild}. Hence, $e(K) = 2$, $c_x = 1$ and furthermore, $d = 3$ by Lemma \ref{lemma:sizeOfKIfHasRelevantChild}.
\end{proof}

\section{Bounding relevant neighbours - the case \boldmath\texorpdfstring{$x \caseTwo y$}{x -2,b-> y}}
\label{sec:caseTwo}
In this section we want to characterize the case where $K$ has two relevant neighbours generated by $(x, x'), (y, y')$ and $x \caseTwo y$ holds. Before that we consider the following lemma, which describes how the components around $K$ will look like in $\barT$ when we are applying Lemma \ref{lemma:newSmallerByPath} in the upcoming sections.

\begin{lemma} \label{lemma:rhoStar}
    Let $K$ be a red component of $\explSG$ and let $C_x, C_y$ be two distinct relevant neighbours of $K$ that are generated by the edges $(x, x')$ and $(y, y')$ of the same tree $\blueTree(\mathcal T^*)$, respectively, where $b \inOneToK$.
    Let $e$ be an edge on the path from $x$ to $y$ in $\redForest(\mathcal T^*)$. 
    Let $F'$ be a forest that can be obtained from $(\redForest(\mathcal T^*) - e) \cup \{xx', yy'\}$ by adding $z'z''$ and removing $z'n_{z'}$ for every $z \in \{x, y\}$ such that $C_z$ is interesting.
    Let $K_z$ be the component in $K - e$ of $z$ for any $z \in \{x, y\}$.
    If $e(K_x) \geq c_y$ and $e(K_y) \geq c_x$, then $\rho(F') = \rho^*$ and the components of $x$ and $y$ in $F'$ each have at most $e(K)$ edges.
\end{lemma}
\begin{proof}
    Let $K'_z$ be the component of $z$ in $F'$ for $z \in \{x, y\}$. If $C_z$ is interesting, then $F'$ contains the component $C'_z$ with $e(C'_z) = e(C_z) - 1$. Further, we have that $e(K'_x) = e(K) - e(K_y) - |\{e\}| + |\{xx'\}| + c_x \leq e(K)$ and similarly $e(K'_y) \leq e(K) - e(K_x) - |\{e\}| + |\{yy'\}| + c_y \leq e(K)$.\\
    Furthermore, $e(K) + |\{xx', yy'\}| + c_x + c_y = e(K'_x) + e(K'_y)  + |\{e\}|$. Thus, if one of the components $K'_x, K'_y$ has $e(K)$ edges, then the other one has $1 + c_x + c_y \leq d$ edges. Hence, $\rho(F') = \rho^*$.
\end{proof}

\begin{lemma} \label{lemma:caseTwo}
    Let $K$ be a red component of $\explSG$ and let $C_x, C_y$ be two distinct relevant neighbours of $K$ generated by $(x, x'), (y, y') \in E(\blueTree(\mathcal T^*))$, respectively, where $b \inOneToK$. Furthermore, suppose we have $x \caseTwo y$ with edge $(u, v)$.\\
    Then $e(K) \geq d$, $c_x = 1$ and the component of $y$ in $K - uv$ has no edges.
\end{lemma}
\begin{proof}
    Let $K_x$ and $K_y$ be the components of $K - uv$ containing $x$ and $y$, respectively. Suppose towards a contradiction that either $e(K) < d$, $c_{x}=0$ or the component of $y$ in $K-uv$ has an edge. Note that if one of the latter two conditions is true, we have $e(K_{y}) \geq c_{x}$. We obtain $\mathcal T'$ from $\mathcal T_x$ by performing $(\bar x, \bar x') \leftrightarrow uv$.
    We have that $\Delta(\mathcal T^*, \mathcal T') = V(\pathInBlueTree{\mathcal T^*}{u}{\bar x})$. Note that all these vertices have a directed blue path towards $y$ in $\blueTree(\mathcal T')$ through $(u, v)$.
    Thus, we can obtain $\mathcal T''$ from $\mathcal T'$ by performing $(y, y') \leftrightarrow n_{y'} y'$ if $C_y$ is interesting. If $C_y$ is small, let $\mathcal T'' := \mathcal T'$. Note that $\mathcal T''$ is depicted in Figure \ref{fig:caseTwo}.
    Let $K'_x$ and $K'_y = K_y$ and be the components in $\redForest(\mathcal T'') + (\bar y, \bar y')$ of $x$ and $y$, respectively.
    We want to obtain a contradiction to Lemma \ref{lemma:newSmallerByPath} and choose $\barT := \mathcal T''$, $(a, a') := (\bar y, \bar y')$. Using the notation from Lemma \ref{lemma:newSmallerByPath}, we have $\Delta(\mathcal T^*, \barT) \subseteq A$ and thus, $B = \varnothing$. Recalling that $i_{B} = \infty$ if $B = \varnothing$, we have that b) holds and it is also clear that a) holds. Further, c) holds since $x, y \in A$. Therefore, it suffices to show that d) holds and $\rho(\redForest(\mathcal T'') + aa') = \rho^*$. We split into cases: \\
    \textbf{Case 1: $e(K) \geq d$}: \\
    In this case, $e(K_y) \geq c_x$. By Lemma \ref{lemma:neighOfXIsDescendant} we have that $e(K_x) \geq 1$ and by Lemma \ref{lemma:rhoStar} we have that $\rho(\redForest(\mathcal T'') + aa') = \rho^*$ and $e(K'_x), e(K'_y) \leq e(K)$. Furthermore, $\bar x, \bar y \in A$ and hence, d) holds, a contradiction.\\
    \textbf{Case 2: $e(K) < d$}:\\
    In this case we have $e(K) = d - 1$ and $c_{x} =1$ by Lemma \ref{lemma:sizeOfKIfHasRelevantChild}. For $z \in \{x, y\}$ we have that $e(K'_z) \leq e(K) - |\{e\}| + 1 + c_z \leq d$ and thus, $\rho(\redForest(\mathcal T'') + aa') = \rho^*$. We split into subcases:\\
    \textbf{Subcase 1: $R^*_{i_A} \neq K$}:\\
    In this case, we immediately obtain that d) holds. \\
    \textbf{Subcase 2: $R^{*}_{i_{A}} =K$ and $w_{i_A}(\sigma^*) \in V(K_y)$}: \\ 
    In this case, there is a legal order $\sigma' = (R'_1, \ldots, R'_{t'})$ for $\mathcal T'$ such that $R'_i = R^*_i$ for all $i < i_A$ and $R'_{i_A} = K_y$. As $e(K_y) < e(K)$, we have $\sigma' < \sigma^*$. Furthermore, $\rho(\redForest(\mathcal T')) = \rho^*$, a contradiction.\\
    \textbf{Subcase 3: $R^{*}_{i_{A}} = K$ and $w_{i_A}(\sigma^*) \in V(K_x)$}:\\ 
    In this case, let $\sigma'' = (R''_1, \ldots, R''_{t''})$ be a legal order of $\mathcal T''$ such that $R''_i = R^*_i$ for all $i < i_A$, $R''_{i_A} = K'_x$, $R''_{i_A + 1} = K'_y - y''$ if $C_y$ is interesting and otherwise $C_{y}$ is small and we let $R''_{i_A + 1} = K'_y$ and $R''_{i_A + 2} = C_y$.
    By the definition of $A$ we have $w_{i_A}(\sigma^*) \notin A$ and thus, $w_{i_A}(\sigma^*) \neq x$. Let $P = (v_0, \ldots, v_l)$ be a minimal special path with respect to $\mathcal T''$, $\sigma''$ and $(\bar y, \bar y')$ that is smaller or equal to the special path from $x$ to $\bar y'$ in $\blueTree(\mathcal T'')$. Note that $v_0$ lies on the red path from $x$ to $w_{i_A}(\sigma^*)$ in $K'_x$. Thus, the component of $v_{-1}$ in $K'_x - x_0 x_{-1}$ has at most $e(K'_x) - |\{\bar x \bar x'\}| - c_x \leq d - 2 < e(K)$ edges, which implies that the legal order we obtain from Lemma \ref{lemma:specialPaths} is smaller than $\sigma^*$, a contradiction.
\end{proof}

\doubleFigure{
    \caseTwoStart
}{
    \caseTwoEnd
}{
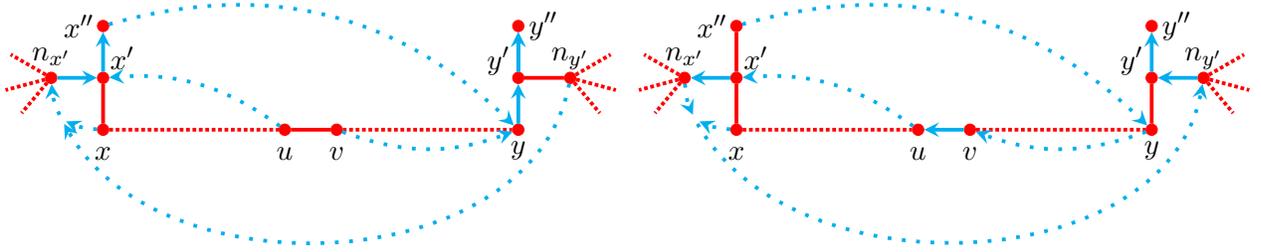
   The component $K$ and its neighbours in $\mathcal T_x$ and $\mathcal T''$ in the proof of Lemma \ref{lemma:caseTwo} with $C_x$ and $C_y$ being interesting. Note that $\pathInBlueTree{\mathcal T_x}{n_{y'}}{y}$ might not visit $x'$ and $x''$ and there might not be a path from $x''$ to $y$ in $\blueTree(\mathcal T'')$.
}{
    \label{fig:caseTwo}
}

\section{Bounding relevant neighbours - the case \boldmath\texorpdfstring{$x \caseOne y$}{x -1,b-> y}}
\label{sec:caseOne}

\singleFigure{
    \caseOneClarificationForVVPrime
}{
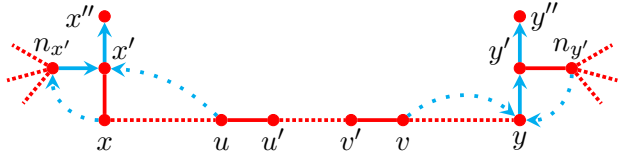
    The component $K$ and its neighbours in $\mathcal T_x$ in the case where $C_x$ and $C_y$ are interesting.
}{
    \label{fig:caseOneClarificationForVVPrime}
}

In this section, let $K$ again be a red component of $\explSG$ and let $C_x, C_y$ be two distinct relevant neighbours of $K$ generated by $(x, x'), (y, y') \in E(\blueTree(\mathcal T^*))$, respectively, where $b \inOneToK$. Furthermore, let $x \caseOne y$ with edges $(u, u')$, $(v', v)$. See Figure \ref{fig:caseOneClarificationForVVPrime} as an illustration.\\
Let $K_x$ be the component of $K - uu'$ containing $x$, and let $K_y$ be the component of $K - vv'$ containing $y$.

    For $v \in V(G)$ let $D^b_v$ denote the set of descendants of $v$ in $\blueTree(\mathcal T_x)$.
    Furthermore, for $v_1, v_2 \in V(G)$ let $D^b_{v_2, v_1} = D^b_{v_2} \setminus D^b_{v_1}$.
    Let $\bar Y$ be the set of vertices having a blue directed path $P$ to $\bar y$ in $\blueTree(\mathcal T_x)$ such that $V(P) \cap D^b_{\bar x} = \varnothing$.\\
    Note that for the following two definitions $b$ is fixed to our choice from the beginning of the section.
    Let $\bar X_{v_1}$ be the set of vertices having a blue directed path $P$ to $\bar x$ in $\blueTree(\mathcal T_x)$ such that $V(P) \cap D^b_{v_1, \bar x} = \varnothing$.\\
    As an illustration of these definitions, note that for the four dotted blue paths belonging to $\blueTree(\mathcal T_x)$ in Figure \ref{fig:caseOneClarificationForVVPrime} we have $V(\pathInBlueTree{\mathcal T_x}{u}{\bar x}), V(\pathInBlueTree{\mathcal T_x}{x}{n_{x'}}) \subseteq D^b_{\bar x} \subseteq \bar X_{\bar y}$ and $V(\pathInBlueTree{\mathcal T_x}{v}{y}), \pathInBlueTree{\mathcal T_x}{n_{y'}}{y}) \subseteq D^b_{\bar y, \bar x} \subseteq \bar Y$.\\
We will end the section with a similar characterization of $K$ and its relevant neighbours like in Lemma \ref{lemma:caseTwo}. An intermediate objective will be Claim \ref{claim:CyInterestingAndNyPathOverXPrime} showing that $C_y$ is  interesting  and that $\pathInBlueTree{\mathcal T_x}{n_{y'}}{y'}$ goes over $(\bar x, \bar x')$.

\begin{claim} \label{claim:wNotInKy}
    If $i_{\sigma^*}(\bar X_{\bar y}) = i_{\sigma^*}(K)$, then $w_{i_{\sigma^*}(K)}(\sigma^*) \in V(K_x)$.
\end{claim}
\begin{proof}
    Suppose to the contrary that $i_{\sigma^*}(\bar X_{\bar y}) = i_{\sigma^*}(K)$ and $w_{i_{\sigma^*}(K)}(\sigma^*) \in V(K) \setminus V(K_x)$.
    We obtain $\mathcal T'$ from $\mathcal T_x$ by performing $(\bar x, \bar x') \leftrightarrow uu'$.
    Note that $\Delta(\mathcal T^*, \mathcal T') \subseteq D^b_{\bar x}$ and thus, $\Delta(\mathcal T^*, \mathcal T') \subseteq \bar X_{\bar y}$. Let $K'_x$ and $K'_y$ be the components of $x$ and $y$ in $\redForest(\mathcal T')$, respectively. We have $e(K'_x) \leq e(K) + 1 + c_x - |\{uu', vv'\}| \leq e(K)$ edges and $e(K'_y) \leq e(K) - |\{uu'\}| < e(K)$. Furthermore, $e(K'_x) + e(K'_y) + |\{uu'\}| = e(K) + |\{xx'\}| + c_x$. Thus, if $e(K'_x) = e(K)$, then $e(K'_y)  = c_x < d$ and thus, $\rho(\redForest(\mathcal T')) = \rho^*$.
    There also exists a legal order $\sigma' = (R'_1, \ldots, R'_{t'})$ with $R'_i = R^*_i$ for all $i < i_{\sigma^*}(K)$ and $R'_{i_{\sigma^*}(K)} = K'_y$. Thus, $\sigma' < \sigma^*$, which is a contradiction.
\end{proof}

\begin{claim} \label{claim:caseOneNormalAugmentOnY}
    We have $i_{\sigma^{*}}(\bar Y) \geq i_{\sigma^{*}}(\bar X_{\bar y})$ if either $C_{y}$ is small, or $C_{y}$ is interesting and $\bar x \notin V(\pathInBlueTree{\mathcal T_x}{n_{y'}}{y'})$. Further, if $C_{y}$ is interesting and  $\bar x \notin V(\pathInBlueTree{\mathcal T_x}{n_{y'}}{y'})$, then there is no path from $u'$ to $y'$ in $\blueTree(\mathcal T_x)$.
\end{claim}
\begin{proof}
    Throughout, if $C_{y}$ is interesting, we assume that $\bar x \notin V(\pathInBlueTree{\mathcal T_x}{n_{y'}}{y'})$. Suppose to the contrary that 
    \begin{enumerate}[1)]
        \item $i_{\sigma^*}(\bar Y) < i_{\sigma^*}(\bar X_{\bar y})$ or
        \item $C_y$ is interesting and there is a path from $u'$ to $y'$ in $\blueTree(\mathcal T_x)$.
    \end{enumerate}
    We obtain $\mathcal T'$ from $\mathcal T_x$ by performing $(\bar x, \bar x') \leftrightarrow uu'$.
    If $C_y$ is interesting, then the path from $n_{y'}$ to $y$ in $\blueTree(\mathcal T_x)$ does not visit a vertex of $\bar X_{\bar y}$ by our assumptions. As $\Delta(\mathcal T_x, \mathcal T') \subseteq  D^b_{\bar x}$, there is (still) a path from $n_{y'}$ to $y$ in $\blueTree(\mathcal T')$ and as $(y, y') \in E(\blueTree(\mathcal T'))$,  there is no path from $y'$ to $y$ in this tree. We obtain $\mathcal T''$ from $\mathcal T'$ by performing $(y, y') \leftrightarrow y'n_{y'}$ if $C_y$ is interesting. If $C_y$ is small, we let $\mathcal T'' := \mathcal T'$. For an illustration of $\mathcal T''$ see Figure \ref{fig:caseOneNormalAugmentOnY}.
    We want to find a contradiction to Lemma \ref{lemma:newSmallerByPath} and choose $\barT := \mathcal T''$, $(a, a') := (\bar y, \bar y')$. In Case 2) we have that $\bar Y \cup \bar X_{\bar y} \subseteq A$ and $B = \varnothing$ since every vertex of $\bar X_{\bar y}$ has a path to $u'$ in $\blueTree(\mathcal T'')$. In Case 1) we have $\bar Y \subseteq A$ and $B \subseteq D^b_{\bar x \subseteq} \bar X_{\bar y}$. Thus, in both cases we have that b) holds and it is also clear that a) holds. Note that in Case 1) we have that $K \neq R^*_{i_A}$ and $K'_x \neq L$. We have $y \in A$ and in Case 2) we also have $x \in A$. Thus, c) holds and in Case 1) we also have that d) holds since $\bar y \in A$. It remains to show that we have $\rho(\redForest(\mathcal T'') + \bar y \bar y') = \rho^*$ and we also have to show that d) holds in Case 2).\\
    First, we suppose that $e(K_x) \geq 1$. Since the component of $y$ in $K - uu'$ also contains an edge $vv'$ we have by Lemma \ref{lemma:rhoStar} that $\rho(\redForest(\mathcal T'') + \bar y \bar y') = \rho^*$ and the components $K'_x$ and $K'_y$ of $x$ and $y$ in $\redForest(\mathcal T'') + \bar y \bar y'$, respectively, have $e(K'_x), e(K'_y) \leq e(K)$. As $\bar x, \bar y \in A$ we have that d) holds in Case 2), which is a contradiction to Lemma \ref{lemma:newSmallerByPath}.\\
    Now, let $e(K_x) = 0$. By Lemma \ref{lemma:neighOfXIsDescendant} we have that $d = 3$, $e(K) = 2$, $c_x = 1$ and thus, $E(K) = \{uu', vv'\}$. We have that $e(K'_x) = 2$, $e(K'_y) = 3$ and thus, $\rho(\redForest(\mathcal T'') + \bar y \bar y') = \rho^*$. In Case 2) we have $x, \bar x, y, \bar y \in A$ and thus, d) holds.
\end{proof}

\doubleFigure{
    \caseOneSimpleOneStart
}{
    \caseOneSimpleOneEnd
}{
    The component $K$ and its neighbours in $\mathcal T_x$ and $\mathcal T''$ in the proof of Claim \ref{claim:caseOneNormalAugmentOnY} in the case where $C_x$ and $C_y$ are interesting and $\pathInBlueTree{\mathcal T_x}{u'}{y'}$ exists.
}{
    \label{fig:caseOneNormalAugmentOnY}
}

\begin{claim} \label{claim:CyInterestingAndNyPathOverXPrime}
    We have that $C_y$ is interesting and $\bar x \in V(\pathInBlueTree{\mathcal T_x}{n_{y'}}{y'})$.
\end{claim}
\begin{proof}
    Suppose that either $C_y$ is small or that $C_y$ is interesting and $\bar x \notin V(\pathInBlueTree{\mathcal T_x}{n_{y'}}{y'})$.
    By Claim \ref{claim:caseOneNormalAugmentOnY} we have that $i_{\sigma^*}(\bar Y) \geq i_{\sigma^*}(\bar X_{\bar y})$ and if $C_y$ is interesting, then there is no path from $u'$ to $y'$ in $\blueTree(\mathcal T_x)$.\\
    If $C_y$ is small, we let $\mathcal T' := \mathcal T_x$ and otherwise we obtain $\mathcal T'$ from $\mathcal T_x$ by performing $(y, y') \leftrightarrow n_{y'}y'$.
    We choose vertex $v_i$ on the path $[v_1, \ldots, v_m]$ from $u'$ to $v$ in $\redForest(\mathcal T')$ such that $i$ is minimal and there is a path in $\blueTree(\mathcal T')$ from $v_i$ to $\bar y$. This vertex exists and $i > 1$ considering that $v$ has a path to $y$ in $\blueTree(\mathcal T')$, and $u'$ does not by our assumptions. We obtain $\mathcal T''$ from $\mathcal T'$ by performing $(\bar y, \bar y') \leftrightarrow v_i v_{i-1}$. The decomposition $\mathcal T''$ is depicted in Figure \ref{fig:CyInterestingAndNyPathOverXPrime}.
    We will now find a contradiction to Lemma \ref{lemma:newSmallerByPath} by letting $\barT := \mathcal T''$ and $(a, a') = (\bar x, \bar x')$. As a consequence we have that $\bar X_{\bar y} \subseteq A$, $B \subseteq D^b_{\bar y, \bar x} \subseteq \bar Y$ and thus, b) holds. It is also clear that a) holds.
    Let $K'_x$ and $K'_y$ be the components of $x$ and $y$, respectively, in $\redForest(\mathcal T'') + \bar x \bar x'$. Note that by Claim \ref{claim:wNotInKy} we have that $L \neq K'_y$.\\
    First, suppose that $|E(K'_y) \cap E(K)| \geq 1$. Since the component of $x$ in $K - v_{i-1}v_i$ contains $uu'$, we have by Lemma \ref{lemma:rhoStar} that $\rho(\redForest(\mathcal T'') + \bar x \bar x') = \rho^*$ and $e(K'_x), e(K'_y) \leq e(K)$. As $\bar x \in A$ and $L \neq K'_y$, we have that c) and d) hold, which is a contradiction.\\
    Now, suppose that $|E(K'_y) \cap E(K)| = 0$. By Lemma \ref{lemma:neighOfXIsDescendant} we have that $d = 3$, $c_y = 1$, $e(K) = 2$ and thus, $E(K) = \{uu', v_{i-1} v_i\}$ and $c_x = 1$. Thus, $e(K'_x) = 3$, $e(K'_y) = 2$ and $\rho(\redForest(\mathcal T'') + \bar x \bar x') = \rho^*$. Note that c) holds since $\bar x \in A$. For the same reason d) holds if $R^*_{i_A} \neq K$.\\
    Thus, suppose that $R^*_{i_A} = K$. By Claim \ref{claim:wNotInKy} we have that $w_{i_A}(\sigma^*) = x$, which is a contradiction since $x \in A$.
\end{proof}

\doubleFigure{
    \caseOneSimpleTwoStart
}{
    \caseOneSimpleTwoEnd
}{
    The component $K$ and its neighbours in $\mathcal T_x$ and $\mathcal T''$ in the proof of Claim \ref{claim:CyInterestingAndNyPathOverXPrime} in the case where $C_x$ and $C_y$ are interesting and $\bar x \notin V(\pathInBlueTree{\mathcal T_x}{n_{y'}}{y'})$.
}{
    \label{fig:CyInterestingAndNyPathOverXPrime}
}

Let $Y$ be the set of vertices having a blue directed path to $y$ in $\blueTree(\mathcal T_x)$ such that $V(P) \cap D^b_{\bar x} = \varnothing$.\\
Let $\mathring Y'$ be the set of vertices having a blue directed path to $y'$ in $\blueTree(\mathcal T_x)$ such that $V(P) \cap D^b_y = \varnothing$.\\
In the proofs of the remaining claims of this section we will implicitly make use of Claim \ref{claim:CyInterestingAndNyPathOverXPrime}.

\begin{claim} \label{claim:caseOneUgly1}
    We have that $i_{\sigma^*}(Y) < i_{\sigma^*}(\bar X_{y'} \cup \mathring Y')$ and additionally, we have that there is no path from $v'$ to $y'$ in $\blueTree(\mathcal T_x)$ or we have that $e(K_y) = 0$, $e(K) \geq d$ and $c_x = 1$.
\end{claim}
\begin{proof}
    Suppose towards a contradiction that $i_{\sigma^*}(Y) \geq i_{\sigma^*}(\bar X_{y'} \cup \mathring Y')$ or there is a path from $v'$ to $y'$ in $\blueTree(\mathcal T_x)$. If this path exists, we also suppose that $e(K_y) \geq c_x$ or $e(K) = d - 1$. Note that if this path does not exist, then we either have $e( K_y) \geq 1$ or $e(K) = d - 1$ by Lemma \ref{lemma:neighOfXIsDescendant}. Thus, we suppose in any case that $e(K_y) \geq c_x$ or $e(K) = d - 1$ holds.
    We obtain $\mathcal T'$ from $\mathcal T_x$ by performing $(y, y') \leftrightarrow vv'$. 
    Note that $\Delta(\mathcal T_x, \mathcal T') \subseteq  D^b_{y, \bar x}$ and thus, $\pathInBlueTree{\mathcal T_x}{n_{y'}}{\bar x'}$ 
    and $\pathInBlueTree{\mathcal T_x}{y'}{r}$ also exist in $\blueTree(\mathcal T')$. Thus, we may obtain $\mathcal T''$ from $\mathcal T'$ by performing $(\bar x, \bar x') \leftrightarrow y'n_{y'}$. The decomposition $\mathcal T''$ is depicted in Figure \ref{fig:caseOneUgly1}. If $\pathInBlueTree{\mathcal T_x}{v'}{y'}$ exists, then it also exists in $\blueTree(\mathcal T'')$ since $\Delta(\mathcal T_x, \mathcal T'') \subseteq  D^b_{\bar x} \cup D^b_{y, \bar x}$.
    We again want to obtain a contradiction to Lemma \ref{lemma:newSmallerByPath} by choosing $\barT := \mathcal T''$ and $(a, a') := (y', y'')$. It is clear that a) holds.
    If $\pathInBlueTree{\mathcal T_x}{v'}{y'}$ exists, then $\Delta(\mathcal T^*, \mathcal T'') \subseteq A$ and thus, $B = \varnothing$. If the path does not exist, then $\bar X_{y'} \cup \mathring Y' \subseteq A$ and $B \subseteq D^b_{y, \bar x} \subseteq Y$. In both cases we have that b) holds. Furthermore, c) holds since $x, y' \in A$.
    Let $K'_x$ and $K'_y$ be the components in $\redForest(\mathcal T'') + y'y''$ of $x$ and $y$, respectively.
    For d) first note that $\bar x, y' \in A$ 
    If $e(K_y) \geq c_x$, then d) holds by Lemma \ref{lemma:rhoStar} and we also have that $\rho(\redForest(\mathcal T'') + aa') = \rho^*$, which is a contradiction.\\
    Thus, suppose that $e(K_y) = 0$, $e(K) = d - 1$ and thus, $c_x = 1$. Then we have that $e(K'_y) = |\{yy', y'y''\}|= 2 < d$, $e(K'_x) = e(K) + |\{xx'\}| + c_x - |\{vv'\}| = d$ and thus, $\rho(\redForest(\mathcal T'') + y'y'') = \rho^*$.
    If $R^*_{i_A} \neq K$, we have that d) holds and obtain a contradiction.\\
    Thus, suppose that $R^*_{i_A} = K$. By Claim \ref{claim:wNotInKy} we have $w_{i_A}(\sigma^*) \in V(K_x)$. Further, there is a legal order $\sigma'' = (R''_1, \ldots, R''_{t''})$ for $\mathcal T''$ such that $R''_i = R^*_i$ for $i < i_A$ and $R''_{i_A} = K'_x$ and $i_{\sigma''}(y') > i_A$. Let $P = [v_0, \ldots, v_l]$ be a minimal special path with respect to $\mathcal T''$, $\sigma''$ and $(y', y'')$ such that $P$ is smaller or equal to the special path $\pathInBlueTree{\mathcal T''}{x}{y''}$. We have that $v_0 \in V(K_x)$.
    Applying 5. of Lemma \ref{lemma:specialPaths}, we receive a decomposition $\mathcal T''' \in \mathcal F^*$ having a legal order $\sigma''' = (R'''_1, \ldots, R'''_{t'''})$ with $R'''_i = R''_i = R^*_i$ for all $i < i_A$ and $R'''_{i_A}$ is the component of $v_{-1}$ having at most $e(K'_x) - c_x - |\{xx', v_0 v_{-1}\}| = d - 3 < e(K)$ edges and obtain $\sigma''' < \sigma^*$, which is a contradiction.
\end{proof}

\doubleFigure{
    \uglyOneStart
}{
    \uglyOneEnd
}{
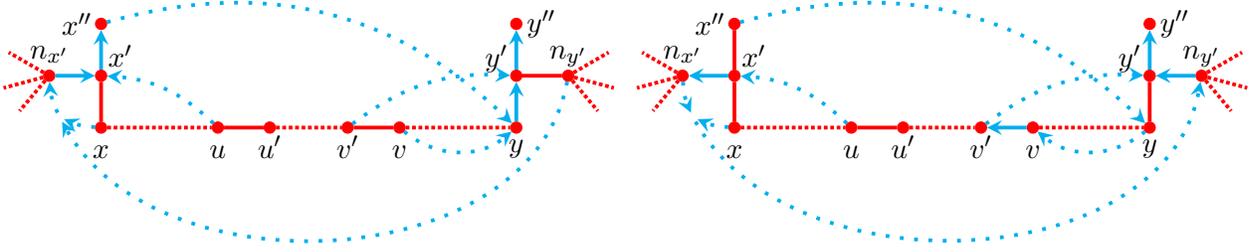
    The component $K$ and its neighbours in $\mathcal T_x$ and $\mathcal T''$ in the proof of Claim \ref{claim:caseOneUgly1} in the case where $C_x$ is interesting and $\pathInBlueTree{\mathcal T_x}{v'}{y'}$ exists.
}{
    \label{fig:caseOneUgly1}
}

\begin{claim} \label{claim:caseOneUgly2}
    There is a path from $u'$ to $y'$ in $\blueTree(\mathcal T_x)$.
\end{claim}
\begin{proof}
    Suppose to the contrary that there is no path from $u'$ to $y'$ in $\blueTree(\mathcal T_x)$.
    We obtain $\mathcal T'$ from $\mathcal T_x$ by performing $(\bar x, \bar x') \leftrightarrow uu'$.
    Note that $\Delta(\mathcal T_x, \mathcal T') \subseteq  D^b_{\bar x}$. In  $\blueTree(\mathcal T')$ there is a path from every vertex of $ D^b_{\bar x}$ to $u'$, which only visits vertices of $D^b_{\bar x} + u'$ and in particular, it does not go over $y'$. Since $u' \notin \Delta(\mathcal T_x, \mathcal T')$, the path from $u'$ to $r$ in $\blueTree(\mathcal T_x)$ also exists in $\blueTree(\mathcal T')$. We conclude that there is no path in $\blueTree(\mathcal T')$ from $n_{y'}$ to $y'$ and thus, we have $(y', y'') \leftrightarrow y' n_{y'}$. We obtain $\mathcal T''$ from $\mathcal T'$ by performing this exchange. Note that $n_{y'}y'$ and $y'y''$ are the only edges affected by this exchange and thus, $\Delta(\mathcal T^*, \mathcal T'') \subseteq  D^b_{\bar x} + y'$. The decomposition $\mathcal T''$ is depicted in Figure \ref{fig:caseOneUgly2}.
    We will again find a contradiction to Lemma \ref{lemma:newSmallerByPath} and choose $\barT := \mathcal T''$, $(a, a') := (y, y')$, which implies $Y \subseteq A$ and $B \subseteq \bar X_{y'} \cup \mathring Y'$. Thus, b) holds by Claim \ref{claim:caseOneUgly1} and it is also clear that a) holds.
    Let $K'_x$, $K'_y$ be the components in $\redForest(\mathcal T'') + yy'$ of $x$ and $y$, respectively. Since $i_{\sigma^*}(Y) < i_{\sigma^*}(\bar X_{y'} \cup \mathring Y')$, we have that $K, C_y \neq R^*_{i_A}$, $L \neq K'_x$ and $L$ is not the component of $y$ in $\redForest(\mathcal T'')$. Thus, c) and d) hold.
    Finally, we have $\rho(\redForest(\mathcal T'') + yy') = \rho^*$ if $e(K_x) \geq 1$ by Lemma \ref{lemma:rhoStar}. If $e(K_x) = 0$, then $d = 3$, $e(K) = 2$ and $c_x = 1$ by Lemma \ref{lemma:neighOfXIsDescendant}. Thus, $e(K'_x) = 2$, $e(K'_y) = 3$ and $\rho(\redForest(\mathcal T'') + yy') = \rho^*$.
\end{proof}

\doubleFigure{
    \uglyTwoStart
}{
    \uglyTwoEnd
}{
   The component $K$ and its neighbours in $\mathcal T_x$ and $\mathcal T''$ in the proof of Claim \ref{claim:caseOneUgly2} in the case where $C_x$ is interesting.
}{
    \label{fig:caseOneUgly2}
}

\begin{claim} \label{claim:caseOneUgly3}
    There is a path from $v'$ to $y'$ in $\blueTree(\mathcal T_x)$.
\end{claim}
\begin{proof}
    Suppose to the contrary that there is no path from $v'$ to $y'$ in $\blueTree(\mathcal T_x)$.
    Let $P$ be the red path in  $\mathcal T_x$ from $u'$ to $v'$. Note that by the definition of $\caseOne$ no vertex of $P$ has a path to $y$ in $\blueTree(\mathcal T_x)$. Let $w'$ be the first vertex on $P$ from which there is not a path in $\blueTree(\mathcal T_x)$ to $y'$. Such a vertex exists and $w' \neq u'$ since $u'$ has a path to $y'$ by Claim \ref{claim:caseOneUgly2} and $v'$ has not. Let $w$ be the vertex in $P$ before $w'$ in $P$ having a path to $y'$.
    We obtain $\mathcal T'$ from $\mathcal T_x$ by performing $(\bar x, \bar x') \leftrightarrow y'n_{y'}$.
    Note that $\Delta(\mathcal T_x, \mathcal T') \subseteq D^b_{\bar x}$ and thus, the paths in $\blueTree(\mathcal T')$ from the vertices of $P$ to $r$ are the same as in $\blueTree(\mathcal T_x)$. Hence, $w$ also has a path to $y'$ in $\blueTree(\mathcal T')$ and $w'$ does not.
    We obtain $\mathcal T''$ from $\mathcal T'$ by performing $(y', y'') \leftrightarrow ww'$. Note that $\mathcal T''$ is depicted in Figure \ref{fig:caseOneUgly3}.
    We want to obtain a contradiction to Lemma \ref{lemma:newSmallerByPath} by choosing $\barT := \mathcal T''$ and $(a, a') := (y, y')$. Note that $Y \subseteq A$ and $B \subseteq D^b_{\bar x} + y' \subseteq \bar X_{y'} \cup \mathring Y'$. Thus, b) holds and it is also clear that a) holds.
    Since $i_{\sigma^*}(\bar x') > i_{\sigma^*}(\bar x), i_{\sigma^*}(C_y') \geq i_{\sigma^*}(\bar X_{y'}) > i_{\sigma^*}(Y)$, we have that $C_x, C_y, K \neq R^*_{i_A}$ and thus, $R^*_{i_A} = L$. Thus, c) and d) hold.
    Finally, since the component in $K - ww'$ of $x$ contains $uu'$ and the other component contains $vv'$ we have $\rho(\redForest(\mathcal T'') +  yy') = \rho^*$ by Lemma \ref{lemma:rhoStar}. 
\end{proof}

\doubleFigure{
    \caseOneUgliestCaseStart
}{
    \caseOneUgliestCaseEnd
}{
    The component $K$ and its neighbours in $\mathcal T_x$ and $\mathcal T''$ in the proof of Claim \ref{claim:caseOneUgly3} in the case where $C_x$ is interesting.
}{
    \label{fig:caseOneUgly3}
}

We summarize the results of the Claims \ref{claim:CyInterestingAndNyPathOverXPrime}, \ref{claim:caseOneUgly1}, \ref{claim:caseOneUgly3} in the following lemma:

\begin{lemma} \label{lemma:caseOne}
    Let $K$ be a red component of $\explSG$ and let $C_x, C_y$ be two distinct relevant neighbours of $K$ generated by $(x, x'), (y, y') \in E(\blueTree(\mathcal T^*))$, respectively, where $b \inOneToK$. Furthermore, let $x \caseOne y$ with edges $(u, u')$, $(v', v)$ and let $K_y$ be the component of $y$ in $K - vv'$.
    
    Then $e(K) \geq d$, $C_y$ is interesting, $\bar x \in V(\pathInBlueTree{\mathcal T_x}{n_{y'}}{y'})$, $e(K_y) = 0$, $c_x = 1$ and there is a path from $v'$ to $y'$ in $\blueTree(\mathcal T_x)$.
\end{lemma}

\section{Bounding the number of relevant neighbours}
\label{sec:threeChildren}

We summarize the results of the last two sections by describing the structure of $K$ if it has two relevant neighbours generated by the same tree:

\begin{lemma} \label{lemma:howTwoChildrenLook}
    Let $K$ be a red component of $\explSG$ and let $C_x, C_y$ be two distinct relevant neighbours of $K$ generated by $(x, x'), (y, y') \in E(\blueTree(\mathcal T^*))$, respectively, where $b \inOneToK$. Without loss of generality let $y$ not be a descendant of $\bar x$ in $\blueTree(\mathcal T^*)$.
    Then $e(K) \geq d$, $c_x = 1$ and one of the following two cases applies:
    \begin{enumerate}[1)]
        \item We have $x \caseOne y$ with edges $(u, u'), (v', v)$, $C_y$ is interesting, $\bar x \in V(\pathInBlueTree{\mathcal T_x}{n_{y'}}{y'})$, there is a path from $v'$ to $y'$ in $\blueTree(\mathcal T_x)$, the component of $y$ in $K - vv'$ has no edges and $x$ is a descendant of $y$ in $\blueTree(\mathcal T^*)$.
        \item We have $x \caseTwo y$ with edge $(u, v)$,  the component of $y$ in $K - uv$ has no edges and $x$ is a descendant of $\bar y$ in $\blueTree(\mathcal T^*)$.
    \end{enumerate}
\end{lemma}
\begin{proof}
    The lemma follows directly from Lemmas \ref{lemma:cases},  \ref{lemma:caseTwo} and \ref{lemma:caseOne} except for the fact that in Case 2) there is a path from $x$ to $\bar y$ in $\blueTree(\mathcal T^*)$. Thus, suppose that 2) holds and no such path exists. Then $u$, which is the only neighbour of $y$, is not a descendant of $y$ in $\blueTree(\mathcal T^*)$ since it has a path to $\bar x$ in $\blueTree(\mathcal T^*)$. By Lemma \ref{lemma:neighOfXIsDescendant} we have that $e(K) = 2 < 3 = d$, which contradicts $e(K) \geq d$.
\end{proof}

The following lemma bounds the number of relevant neighbours of $K$. This also bounds the number of small children of $K$, which will help us to show that the density around $K$ is high.

\begin{lemma} \label{lemma:numRelevantNeighbours}
    Let $K$ be a red component of $\explSG$. Then $K$ has at most two relevant neighbours generated by $\blueTree(\mathcal T^*)$ for any $b \inOneToK$.
\end{lemma}
\begin{proof}
    Suppose to the contrary that there are three relevant neighbours of $K$ generated by $(x, x')$, $(y, y')$, $(z, z') \in E(\blueTree(\mathcal T^*))$, respectively.
    By Lemma \ref{lemma:howTwoChildrenLook} we may assume that there is a path in $\blueTree(\mathcal T^*)$ from $x$ over $\bar y$ to $\bar z$ and that $y$ and $z$ only have one neighbour in $K$, $n_y$ and $n_z$, respectively.

    \begin{itemize}
        \item If we have $x \caseOne y$, then $C_y$ is interesting and we obtain $\mathcal T'_x$ from $\mathcal T_x$ by performing $(\bar x, \bar x') \leftrightarrow n_{y'} y'$. This is possible because $\pathInBlueTree{\mathcal T_x}{n_{y'}}{y'}$ goes over $(\bar x, \bar x')$. Note that $\pathInBlueTree{\mathcal T_x}{n_y}{y'}$ (still) exists in  $\blueTree(\mathcal T'_x)$ since all vertices of $\Delta(\mathcal T_x, \mathcal T'_x)$ have a path to $y$ in $\blueTree(\mathcal T_x)$ and $n_y$ does not. Thus, we may obtain $\mathcal T_1$ from $\mathcal T'_x$ by performing $(y, y') \leftrightarrow n_y y$.\\
        If we have $x \caseTwo y$, we obtain $\mathcal T'_x$ from $\mathcal T_x$ by performing $(\bar x, \bar x') \leftrightarrow y n_y$. Note that if $C_y$ is interesting,
        there (still) is a path from $n_{y'}$ to $y$ in $\blueTree(\mathcal T'_x)$ since every vertex of $\Delta(\mathcal T_x, \mathcal T'_x)$ has a path to $y$ in $\blueTree(\mathcal T'_x)$ over $(n_y, y)$.
        Thus, we may obtain $\mathcal T_1$ from $\mathcal T'_x$ by performing $(y, y') \leftrightarrow n_{y'} y'$. If $C_y$ is small, we let $\mathcal T_1 := \mathcal T'_x$.
        \item We perform the exchanges of the previous bullet point but for $y$ and $z$ instead of $x$ and $y$. For this, note that every vertex of $\Delta(\mathcal T_x, \mathcal T_1)$ has a path to $\bar y$ in $\blueTree(\mathcal T_x)$, as well as in $\blueTree(\mathcal T_1)$. Thus, the set of all descendants of $\bar y$ in $\blueTree(\mathcal T_1)$ is the same as in $\blueTree(\mathcal T_x)$. The same holds for $z$ since it is an ancestor of $y$ in $\blueTree(\mathcal T_x)$, as well as in $\blueTree(\mathcal T_1)$. Thus, the following exchanges are performable with the same reasoning as in the previous bullet point.\\
        If we have $y \caseOne z$, we obtain $\mathcal T_2$ from $\mathcal T_1$ by first performing $(\bar y, \bar y') \leftrightarrow n_{z'} z'$ and then performing $(z, z') \leftrightarrow n_z z$.\\
        If we have $y \caseTwo z$, we obtain $\mathcal T'_1$ from $\mathcal T_1$ by performing $(\bar y, \bar y') \leftrightarrow z n_z$. If $C_z$ is small, we let $\mathcal T_2 := \mathcal T'_1$, otherwise we obtain $\mathcal T_2$ from $\mathcal T'_1$ by performing $(z, z') \leftrightarrow n_{z'} z'$.
    \end{itemize}
    Note that $\mathcal T_2$ is depicted in Figure \ref{fig:numRelevantNeighbours}.
    We want to obtain a contradiction to Lemma \ref{lemma:newSmallerByPath} and choose $\bar T := \mathcal T_2$ and $(a, a') := (\bar z, \bar z')$. Note that we have $\Delta(\mathcal T^*, \mathcal T_2) \subseteq A$ and thus, $B = \varnothing$. It is clear that a) and b) hold. Let $K'_x, K'_y, K'_z$ be the components in $\redForest(\mathcal T_2) + \bar z \bar z'$ of $x, y, z$, respectively. We have that $e(K'_y) = |\{yy'\}| + c_y  < d \leq e(K)$, $e(K'_z) = |\{zz'\}| + c_z < d \leq e(K)$ and $e(K'_x) = e(K) - |\{n_y y, n_z z\}| + |\{xx'\}| + c_x = e(K)$. Thus, $\rho(\redForest(\mathcal T_2) + \bar z \bar z') = \rho^*$. Furthermore, c) holds since $x, y, z \in A$ and finally, d) holds since for all $v \in \{x, y, z\}$ we have that $v' \in A$ if $C_v$ is not small.
\end{proof}

\doubleFigure{
    \threeChildrenStart
}{
    \threeChildrenEnd
}{
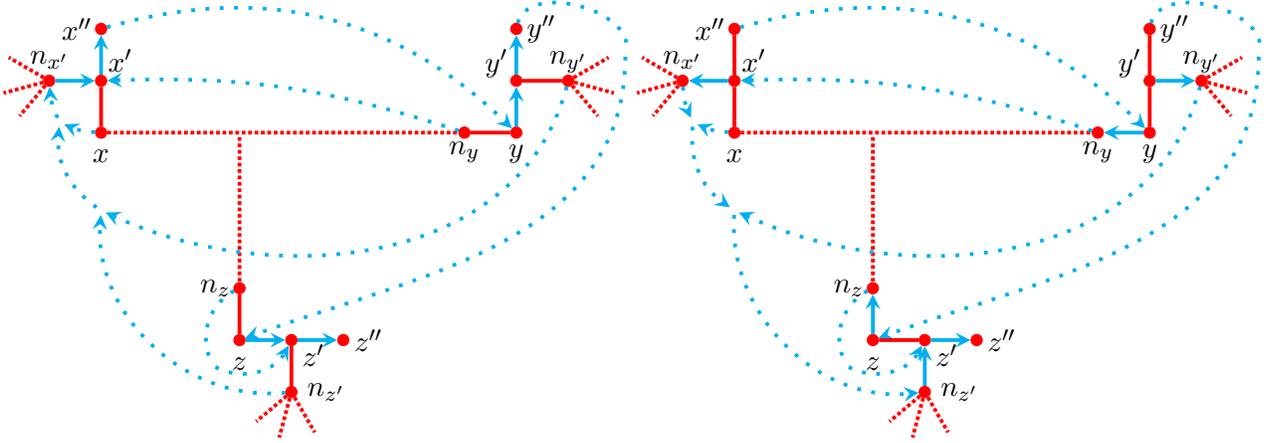
   The component $K$ and its neighbours in $\mathcal T_x$ and $\mathcal T_2$ in the proof of Lemma \ref{lemma:numRelevantNeighbours} in the case where $x \caseTwo y$, $y \caseOne z$ and $C_x, C_y$ are interesting.
}{
    \label{fig:numRelevantNeighbours}
}

\section{Bad components are interesting neighbours}
\label{sec:badIsInteresting}

Note that in the case of $x \caseTwo y$ it could be that $(x, x')$ generates a child with exactly one edge and $(y, y')$ generates a child without any edges. If several blue spanning trees generate two such children for $K$ and if $f$ assigned these children to $K$, then $\frac{e(K) + \sum_{C \in f^{-1}(K)} e(C)}{v(K) + \sum_{C \in f^{-1}(K)} v(C)}$ might be smaller than $d/(d+k+1)$. Thus, we want to show that such a ``bad'' component is an interesting neighbour of some other component $L$ and we can send the child with one edge to $L$. If $L$ also has a child with zero edges, we want to call it bad as well and, using the same technique as for $K$, send the child of $K$ with one edge to another component until it is received by a component that does not have many small children. First, we formalize in which case we call a component bad.

\begin{definition} \label{def:bad}
    Let $K$ be a red component of $\explSG$ having two relevant neighbours $C_x$ and $C_y$ generated by $(x, x'), (y, y') \in E(\blueTree(\mathcal T^*))$, respectively. Furthermore, we let $e(C_y) = 0$ and thus, we have $x \caseTwo y$ by Lemma \ref{lemma:howTwoChildrenLook}. Then we say that \textit{$K$ is $b$-bad (because of $x$ and $y$)}. 
\end{definition}

For the rest of this section let $K$ be $b$-bad because of $x$ and $y$ as in Definition \ref{def:bad}. Thus, $y$ has exactly one neighbour $n_y$ in $K$.
In this section we aim to show that the last arc $(z, y)$ of the path from $n_y$ to $y$ in $\blueTree(\mathcal T^*)$ has its tail outside of $K$. This would imply that $K$ is an interesting neighbour of the red component of $z$. 
For the rest of the section we suppose to the contrary that $z \in V(K)$.

\begin{claim} \label{claim:predXCaseOneZ}
  We have that $x \caseOne z$.
\end{claim}
\begin{proof}
    Note that there is a path from $\bar x$ to $z$ in $\blueTree(\mathcal T^*)$ (and thus also in $\blueTree(\mathcal T_x)$) and hence, $z$ is not a descendant of $\bar x$ in $\blueTree(\mathcal T_x)$. Thus, suppose to the contrary that $x \caseTwo z$ holds with edge $(u, v)$ 
    We obtain $\mathcal T'$ from $\mathcal T_x$ by performing $(\bar x, \bar x') \leftrightarrow uv$. Note that every vertex of $\Delta(\mathcal T^*, \mathcal T')$ has a path to $z$ in $\blueTree(\mathcal T^*)$ and in $\blueTree(\mathcal T')$. Thus, there (still) is a path from $n_y$ to $z$ in $\blueTree(\mathcal T')$ and $(z, y) \in E(\blueTree(\mathcal T'))$. We obtain $\mathcal T''$ from $\mathcal T'$ by performing $(z, y) \leftrightarrow n_y y$. Note that this does not create a red cycle since $y$ (still) has degree one in $\redForest(\mathcal T'')$. Note that $\mathcal T''$ is depicted in Figure \ref{fig:predXCaseOneZ}.
    We want to obtain a contradiction to Lemma \ref{lemma:newSmallerByPath} by choosing $\barT := \mathcal T''$ and $(a, a') := (y, y')$, which gives us $\Delta(\mathcal T^*, \mathcal T'') \subseteq A$ and $B = \varnothing$. It is clear that a) and b) hold. For c) note that $\bar x, y \in A$. Let $K'_x$ be the component of $x$ and $K'_y$ be the component of $y$ and $z$ in $\redForest(\mathcal T'') + yy'$.
    We have that $e(K'_x) \leq e(K) - |\{uv, n_y y\}| + |\{xx'\}| + c_x = e(K)$ and $e(K'_y) \leq e(K) - |\{uv, n_y y\}| + |\{zy, yy'\}| = e(K)$. Furthermore, we have that 
    \[e(K) + |\{zy, yy', xx'\}| + c_x = e(K'_x) + e(K'_y) + |\{uv, n_y y\}|.\]
    Thus, if one of the components $K'_x, K'_y$ has $e(K)$ edges, then the other one has only $2  < d$ edges. Thus, $\rho(\redForest(\mathcal T'') + yy') = \rho^*$.
    Furthermore, d) holds since $x' \in A$, which leads to a contradiction to Lemma \ref{lemma:newSmallerByPath}.
\end{proof}

\doubleFigure{
    \predTwoStart
}{
    \predTwoEnd
}{
   The component $K$ and its neighbours in $\mathcal T_x$ and $\mathcal T''$ in the proof of Claim \ref{claim:predXCaseOneZ} in the case where $C_x$ is interesting.
}{
    \label{fig:predXCaseOneZ}
}

For the rest of the section suppose that $x \caseOne z$ holds  with edges $(u, u')$ and $(v', v)$.
Let $K_z$ be the component of $z$ in $K - v'v$.

\begin{claim} \label{claim:predKzGeq1}
    $e(K_z) \geq 1$.
\end{claim}
\begin{proof}
    Suppose towards a contradiction that $e(K_z) = 0$ and thus, $z = v$. We obtain $\mathcal T'$ from $\mathcal T_x$ by performing $(z, y) \leftrightarrow zv'$. Note that the exchange only affects these two edges and $\mathcal T'$ does not contain a red cycle because $z$ (still) has degree $1$ in $\redForest(\mathcal T')$. Thus, if we replace $K$ by $K - zv' + zv$ in $\sigma^*$ we obtain a legal order for $\mathcal T'$ that is lexicographically equal to $\sigma^*$. Thus, Lemma \ref{lemma:caseTwo} also applies to $\mathcal T'$ and $\sigma'$, which is a contradiction since the component of $y$ in $\redForest(\mathcal T') - n_y y$ contains $yz$.
\end{proof}

We maintain the definitions of $D^b_{\bar x}$ and $D^b_{y, \bar x}$ for $\mathcal T_x$ of Section \ref{sec:caseOne}. Furthermore, let $Z$ be the set of vertices having a blue directed path $P$ to $z$ in $\blueTree(\mathcal T_x)$ such that $V(P) \cap D^b_{\bar x} = \varnothing$ and let $\bar X_z$ still be the set of vertices having a blue directed path $P$ to $\bar x$ in $\blueTree(\mathcal T_x)$ such that $V(P) \cap D^b_{z, \bar x} = \varnothing$.

\begin{claim} \label{claim:predXLeqY}
    $i_{\sigma^*}(\bar X_z) \leq i_{\sigma^*}(Z)$.
\end{claim}
\begin{proof}
    Suppose to the contrary that $i_{\sigma^*}(\bar X_z) > i_{\sigma^*}(Z)$. 
    We obtain $\mathcal T'$ from $\mathcal T_x$ by performing $(\bar x, \bar x') \leftrightarrow uu'$. Note that $\mathcal T'$ is depicted in Figure \ref{fig:predXLeqY}.
    We want to obtain a contradiction to Lemma \ref{lemma:newSmallerByPath} by choosing $\barT := \mathcal T'$ and $(a, a') := (y, y')$. This implies $Z \subseteq A$ and $B \subseteq D^b_{\bar x} \subseteq \bar X_z$. This implies b) and it is also clear that a) holds.
    Let $K'_x$ and $K'_z$ be the components of $x$ and $z$ in $\redForest(\mathcal T') + yy'$. Note that it is not clear whether $y$ is contained in $K'_x$ or $K'_z$.
    We have that $e(K'_x) \leq e(K) + c_x + |\{xx', yy'\}| - |\{uu', vv'\}| - e(K_z) \leq e(K)$ by Claim \ref{claim:predKzGeq1} and $e(K'_z) \leq e(K) - |\{uu'\}| + |\{yy'\}| \leq e(K)$. Furthermore,
    \[e(K) + |\{xx', yy'\}| + c_x = e(K'_x) + e(K'_z) + |\{uu'\}|.\]
    Thus, if one of the components $K'_x, K'_z$ contains $e(K)$ edges, then the other one contains $2 < d$ edges. Thus, $\rho(\redForest(\mathcal T') + yy') = \rho^*$.
    Furthermore, we have that $L \neq K'_x$ since $i_{\sigma^*}(K'_x) \geq i_{\sigma^*}(\bar X_z) > i_{\sigma^*}(Z) \geq i_A$. Thus, d) holds and c) holds as well considering that $z \in A$.
\end{proof}

\doubleFigure{
    \predOneStart
}{
    \predOneEndOne
}{
   The component $K$ and its neighbours in $\mathcal T_x$ and $\mathcal T'$ in the proof of Claim \ref{claim:predXLeqY} in the case where $C_x$ is interesting.
}{
    \label{fig:predXLeqY}
}

We now state the main result of this section, where we reintroduce notation for clarity:
\begin{lemma} \label{lemma:predNotinK}
    Let $K$ be a $b$-bad component because of $x$ and $y$ and let $z$ be the vertex before $y$ on the path from $n_y$ to $y$ in $\blueTree(\mathcal T^*)$.
    Then $z \notin V(K)$. 
\end{lemma}
\begin{proof}
    Suppose that $z \in V(K)$. Then by Claim \ref{claim:predXCaseOneZ}, \ref{claim:predKzGeq1} and \ref{claim:predXLeqY} we have that $x \caseOne z$ with edges $(u, u')$ and $(v', v)$, $e(K_z) \geq 1$ and $i_{\sigma^*}(\bar X_z) \leq i_{\sigma^*}(Z)$.
    We obtain $\mathcal T'$ from $\mathcal T_x$ by performing $(\bar x, \bar x') \leftrightarrow n_y y$. Note that $\Delta(\mathcal T_x, \mathcal T') \subseteq \bar X_z$ and thus, we may obtain $\mathcal T''$ from $\mathcal T'$ by performing $(z, y) \leftrightarrow vv'$. The decomposition $\mathcal T''$ is depicted in Figure \ref{fig:predNotinK}.
    We want to obtain a contradiction to Lemma \ref{lemma:newSmallerByPath} and choose $\barT := \mathcal T''$ and $(a, a') := (y, y')$. This implies $\bar X_z \subseteq A$ and $B \subseteq D^b_{z, \bar x} \subseteq Z$. Thus, b) holds and it is also clear that a) holds. c) holds as well since $x, y \in A$. Let $K'_x$ be the component of $x$ and $K'_y$ be the components of $y$ and $z$ in $\redForest(\mathcal T'') + yy'$. We have that 
    $e(K'_x) = e(K) - |\{n_y y, vv'\}| - e(K_z) + |\{xx'\}| + c_x  < e(K)$
    and
    $e(K'_y) \leq e(K) - |\{vv', n_y y, uu'\}| + |\{yy', zy\}| < e(K)$.
    Thus, $\rho(\redForest(\mathcal T') + yy') = \rho^*$.
    Finally, we have that d) holds since $x' \in A$ and we obtain a contradiction.
\end{proof}

\doubleFigure{
    \predOneStart
}{
    \predOneEndTwo
}{
   The component $K$ and its neighbours in $\mathcal T_x$ and $\mathcal T''$ in the proof of Lemma \ref{lemma:predNotinK} in the case where $C_x$ is interesting.
}{
    \label{fig:predNotinK}
}

\section{Defining \boldmath\texorpdfstring{$f$}{f}}
\label{sec:defineF}

In this final section we will finally define $f$ and prove Lemma \ref{lemma:densityOfKC}. First, we will show to which component of $\mathcal K$ a child of a $b$-bad component generated by $\blueTree(\mathcal T^*)$ containing exactly one edge will be assigned to. This is $K_n$ in the following definition, when the sequence is a sink sequence.

\begin{definition} \label{def:sinkSeq}
    Let $n \geq 1$ and $b \inOneToK$. Let $K_1, \ldots, K_n$ be red components of $\explSG$ and $(x_i, x'_i) \in E(\blueTree(\mathcal T^*))$ for $i \in \{1, \ldots, n\}$. We call $(K_1, x_1), \ldots, (K_n, x_n)$ a \textit{partial sink sequence for $b$} if the following conditions are met:
    \begin{enumerate}[1)]
        \item $K_1$ is $b$-bad because of $x_1$ and $x'_2$ such that the relevant neighbour $C_{x_1}$ of $K_1$ that is generated by $(x_i, x'_i) \in E(\blueTree(\mathcal T^*))$ has $e(C_{x_1}) = 1$.
        \item For $i \in \{2, \ldots, n\}$ we have that $K_{i-1}$ is an interesting neighbour of $K_i$ generated by $(x_i, x'_i) \in E(\blueTree(\mathcal T^*))$.
    \end{enumerate}
    Furthermore, we call $(K_1, x_1), \ldots, (K_n, x_n)$ a \textit{sink sequence for $b$} if it is a partial sink sequence for $b$ and additionally, $K_n$ is not $b$-bad.
\end{definition}

Let us make some critical observations about partial sink sequences. First, note that in a sink sequence we have that $K_1 \neq K_n$ and thus, $n \geq 2$.

\begin{obs} \label{obs:sinkSequenceMostCompsBad}
    In a partial sink sequence $(K_1, x_1), \ldots, (K_n, x_n)$ for any $i \in \{1, \ldots, n-1\}$ the component $K_i$ is $b$-bad due to $x_i$ and $x'_{i+1}$.
\end{obs}

\begin{obs} \label{obs:sinkSequenceDistinctComps}
    In a partial sink sequence $(K_1, x_1), \ldots, (K_n, x_n)$ the components $K_1, \ldots, K_n$ are pairwise distinct.
\end{obs}
\begin{proof}
    Let $i \in \{1, \ldots, n-1\}$. Note that since $K_i$ is $b$-bad due to $x_i$ and $x'_{i+1}$ by Observation \ref{obs:sinkSequenceMostCompsBad}, there is a path in $\blueTree(\mathcal T^*)$ from $n_{x'_{i+1}}$ to $x'_{i+1}$ visiting $\bar x_i$ and then $x_{i+1}$. Thus, $x_i$ is a proper descendant of $x_{i+1}$ in $\blueTree(\mathcal T^*)$ and the claim follows.
\end{proof}

\begin{obs}
\label{partialsinksequenceextension}
If $(K_{1},x_{1}),\ldots,(K_{i},x_{i})$ is a partial sink sequence for $b \inOneToK$ which is not a sink sequence, then there exists a pair $(K_{i+1},x_{i+1})$ such that $(K_{1},x_{1}),\ldots,(K_{i},x_{i}),(K_{i+1},x_{i+1})$ is a partial sink sequence.
\end{obs}
\begin{proof}
Let $(K_{1},x_{1}), \ldots, (K_{i},x_{i})$ be a partial sink sequence that is not a sink sequence. Thus, $K_i$ is $b$-bad due to $x_i$ and another vertex $y$. By Lemma \ref{lemma:predNotinK} the vertex $z$ on the path from $n_{y}$ to $y$ in $\blueTree(\mathcal T^*)$ which is adjacent to $y$ is not in $K_i$. Let $K_{i+1}$ be the component containing $z$. Then $K_i$ is an interesting neighbour of $K_{i+1}$ and thus, $(K_{1},x_{1}),\ldots,(K_{i},x_{i}),(K_{i+1},z)$ is a partial sink sequence.
\end{proof}

\begin{corollary}
Every partial sink sequence extends to a sink sequence.
\end{corollary}

\begin{lemma}
    Let $b \inOneToK$ and let $L$ be a red component of $\explSG$ that is not $b$-bad having $\ell$ interesting neighbours generated by $\blueTree(\mathcal T^*)$. 
    Then $L$ is contained in at most $\ell$ sink sequences for $b$ (and always is the end of such a sequence).
\end{lemma}
\begin{proof}
Since $L$ is not $b$-bad, it follows that $L$ can only be the end of any sink sequence for $b$. To see that $L$ belongs in at most $\ell$ sink sequences for $b$, it suffices to show that
for every vertex $x \in V(L)$ such that $L$ has an interesting neighbour generated by an arc $(x, x') \in E(\blueTree(\mathcal T^*))$, there is at most one sink sequence for $b$ ending with the tuple $(L, x)$. Thus, suppose that there are two sink sequences $(K_1,x_1),\ldots,(K_{n-1},x_{n-1}),(K_n,x_n)$, and $(K'_{1}, y_{1}),\ldots,(K'_{m-1},y_{m-1}),(K'_m,y_m)$, where $(K_n, x_n) = (L, x) = (K'_m, y_m)$, and
suppose that $(K_{n-i}, x_{n-i}) = (K'_{m-i}, y_{m-i}), \ldots, (K_n, x_n) = (K'_m, y_m)$ for some $i \in \{0, \ldots, \min \{n-2, m-2\}\}$ and let $x'_j$ and $y'_j$ be the parents of $x_j$ and $y_j$ in $\blueTree(\mathcal T^*)$, respectively.
\begin{claimUnnumbered}
        $(K_{n-(i+1)}, x_{n-(i+1)}) = (K'_{m-(i+1)}, y_{m-(i+1)})$ 
\end{claimUnnumbered}
\begin{proofInProof}
    Clearly, $x'_{n-i} = y'_{m-i}$ and thus, $K_{n-(i+1)} = K'_{m-(i+1)}$.
    Note that $K_{n-(i+1)}$ is $b$-bad due to $x_{n-(i+1)}$ and $x'_{n-i}$ and $K'_{m-(i+1)}$ is $b$-bad due to $y_{m-(i+1)}$ and $y'_{m-i}$ by Observation \ref{obs:sinkSequenceMostCompsBad}. By Lemma \ref{lemma:numRelevantNeighbours}, $K_{n-(i+1)}$ does not have more relevant neighbours generated by $\blueTree(\mathcal T^*)$ than the two that are generated by $(x_{n-(i+1)}, x'_{n-(i+1)})$ and $(x'_{n-i}, x''_{n-i}) = (y'_{m-i}, y''_{m-i})$. Thus, $(x_{n-(i+1)}, x'_{n-(i+1)}) = (y_{m-(i+1)}, y'_{m-(i+1)})$.
\end{proofInProof}
From the claim it follows (without loss of generality) that $(K_1, x_1) = (K'_{m-(n-1)}, y_{m-(n-1)}), \\ \ldots, (K_n, x_n) = (K'_m, y_m)$. To prove the lemma it only remains to show that $n = m$ and thus, the sink sequences are equal.
This follows directly from the fact that $x'_1$ is contained in a red component having exactly one edge and $x'_1 = y'_{m-(n-1)}$. Thus, the component $K'_{m-(n-1)}$ containing $y_{m-(n-1)}$ can only be in a first tuple of a sink sequence.
\end{proof}

Putting these together we obtain the following:
\begin{lemma}
    \label{lemma:uniquenessOfSinkSequences}
    Let $b \inOneToK$. Let $K$ be a red component of $\explSG$ that is $b$-bad because of $x$ and $y$ and let $C_x$ be the relevant neighbour of $K$ generated by $(x, x') \in E(\blueTree(\mathcal T^*))$.
    Then $K$ is contained in exactly one sink sequence for $b$ and if $e(C_x) = 1$, then $K$ is the beginning of this sink sequence.\\
    Let $L$ be a red component of $\explSG$ that is not $b$-bad having $\ell$ interesting neighbours generated by $\blueTree(\mathcal T^*)$. 
    Then $L$ is contained in at most $\ell$ sink sequences for $b$ (and always is the end of such a sequence).
\end{lemma}

    Now we are in position to define $f$. Recall that $\mathcal{C}$ is the set of small components, and $\mathcal{K}$ is the set of red components that are not $R^{*}$ nor small.

\begin{lemma}  \label{lemma:assignment}
    There is a function $f: \mathcal C \longrightarrow \mathcal K$ such that for every $K \in \mathcal K$ we have:
    \begin{itemize}
        \item if $e(K) < d-1$, then $f^{-1}(K) = \varnothing$.
        \item if $e(K) = d-1$, then $|f^{-1}(K)| \leq k$ and each of the components of $f^{-1}(K)$ has exactly one edge.
        \item if $e(K) \geq d$, then there are integers $q_0, q_1 \geq 0$ with $q_0 + q_1 \leq k$ such that $f^{-1}(K)$ contains exactly $q_0$ components having zero edges and exactly $2q_1$ components having one edge.
    \end{itemize}
\end{lemma}
\begin{proof}
    Let $C$ be a small component. By Lemma \ref{lemma:sizeOfKIfHasRelevantChild} we have that $C$ has some parent $K \in \mathcal K$ with respect to $\mathcal T^*$ and $\sigma^*$ having at least $d - e(C)$ edges. Let $C$ be generated by $(x, x') \in E(\blueTree(\mathcal T^*))$. 
    We assign $C$ to $K$ except if $e(C) = 1$ and $K$ has another small child $C'$ generated by $\blueTree(\mathcal T^*)$ with $e(C') = 0$ and hence, $K$ is $b$-bad: In this case, by Lemma \ref{lemma:uniquenessOfSinkSequences}, there exists a sink sequence $(K_1, x_1), \ldots, (K_n, x_n)$ for $b$ where $(K_1, x_1) = (K, x)$ and we assign $C$ to $K_n$. Note that by Lemma \ref{lemma:sizeOfKIfHasRelevantChild} we have $e(K_n) \geq d - 1$.\\ 
    With this definition of $f$ note that by Lemma \ref{lemma:uniquenessOfSinkSequences} the number of small children with zero edges assigned to some $K \in \mathcal K$ is at most the number of small children of $K$ with respect to $\mathcal T^*$ and $\sigma^*$ not having an edge and thus bounded by $k$ by Lemma \ref{lemma:howTwoChildrenLook}. Further, the number of small children with one edge assigned to some $K \in \mathcal K$ is at most the number of relevant neighbours of $K$  containing at least one edge, which is at most $2k$ by Lemma \ref{lemma:numRelevantNeighbours} (and thus, we obtain the desired integers $q_{0}$ and $q_{1}$).
    The lemma follows by Lemma \ref{lemma:sizeOfKIfHasRelevantChild}, \ref{lemma:howTwoChildrenLook} and \ref{lemma:numRelevantNeighbours}.
\end{proof}

Finally, we only need to show that this definition of $f$ suffices to prove Lemma \ref{lemma:densityOfKC}.

\begin{proofOfFinalLemma}
    The proof is straightforward using  function $f$ of Lemma \ref{lemma:assignment}.
    First, let $e(K) < d - 1$. Then $f^{-1}(K) = \varnothing$. As $K$ is not small, we have $e(K) \geq 2$ and thus, 
    \[
        \frac{e(K) + \sum_{C \in f^{-1}(K)} e(C)}{v(K) + \sum_{C \in f^{-1}(K)} v(C)} 
        \geq \frac{2}{3}
        \geq \frac{d}{d+k+1}.
    \]
    Next, let $e(K) = d - 1 (\geq 2)$. Then
    \[
    \frac{e(K) + \sum_{C \in f^{-1}(K)} e(C)}{v(K) + \sum_{C \in f^{-1}(K)} v(C)}
    = \frac{d-1 + |f^{-1}(K)| \cdot 1}{d + |f^{-1}(K)| \cdot 2}
    \geq \frac{d + (k-1)}{d + k + 1 + (k-1)}
    \geq \frac{d}{d+k+1}.
    \]
    Finally, let $e(K) \geq d$ and let $q_0$ and $q_1$ be defined as in Lemma \ref{lemma:assignment}. Then
    \[
    \frac{e(K) + \sum_{C \in f^{-1}(K)} e(C)}{v(K) + \sum_{C \in f^{-1}(K)} v(C)}
    = \frac{e(K) + 2q_1}{e(K) + 1 + 2\cdot 2q_1 + q_0}
    \geq \frac{d + 2q_1}{d + k + 1 + 3q_1}
    \geq \frac{d}{d+k+1}.
    \]
\end{proofOfFinalLemma}

\begin{ack}
The authors thank Matthew Kwan and ISTA for hosting Sebastian Mies where a portion of this work was completed. We thank anonymous referees for feedback, which improved this work.
\end{ack}

\printbibliography

\end{document}